\documentclass[12pt]{amsart}
\usepackage{a4wide, amsmath, amssymb}
\usepackage[usenames]{color}
\newtheorem{thm}{Theorem}[section]

\newtheorem{lem}[thm]{Lemma}
\newtheorem{prop}[thm]{Proposition}
\theoremstyle{definition}
\theoremstyle{definition}
\theoremstyle{definition}
\newtheorem{ex}[thm]{Example}\theoremstyle{definition}

\newtheorem{rem}[thm]{Remark} \numberwithin{equation}{section}

\newcommand{\R}{\mathbb R}

\newcommand{\T}{\mathbb T}

\def\1{\mathbb I}

\def\e{\epsilon}

\begin{document}
\title[Large time behavior for nonlinear degenerate parabolic equations]{Large time behavior for some 
nonlinear degenerate parabolic equations}
\date{\today}

\author{Olivier Ley\and Vinh Duc Nguyen}
\address{IRMAR, INSA de Rennes, 35708 Rennes, France} \email{olivier.ley@insa-rennes.fr}
\address{IRMAR, INSA de Rennes, 35708 Rennes, France} \email{vinh.nguyen@insa-rennes.fr}

\begin{abstract}
We study the asymptotic behavior of Lipschitz
continuous solutions
of nonlinear degenerate parabolic equations in the
periodic setting. Our results apply to a large
class of Hamilton-Jacobi-Bellman equations. Defining $\Sigma$ as the set where
the diffusion vanishes, i.e., where the equation is totally
degenerate, we obtain the convergence when the equation is
uniformly parabolic outside $\Sigma$ and, on $\Sigma,$ the Hamiltonian
is either strictly convex or satisfies an assumption
similar of the one introduced by Barles-Souganidis (2000) for
first-order Hamilton-Jacobi equations. This latter assumption
allows to deal with equations with nonconvex Hamiltonians.
We can also release the uniform parabolic requirement
outside $\Sigma$.
As a consequence, we prove the convergence of some everywhere degenerate
second-order equations.
\end{abstract}


\maketitle

\section{Introduction}

The large time behavior of the solution of
\begin{eqnarray}\label{DHJE-intro}
&& \left\{
\begin{array}{ll}
\displaystyle\frac{\partial u}{\partial t}
+\mathop{\rm sup}_{\theta\in\Theta}
\{-{\rm trace}(A_\theta(x)D^2 u) + H_\theta(x,Du)\}=0, &
(x,t)\in \T^N\times (0,+\infty),\\[3mm]
u(x,0)=u_{0}(x), & x\in\T^N,
\end{array}
\right.
\end{eqnarray}
in the periodic setting ($\T^N$ is the flat torus)
was extensively studied (see the references below) in two
frameworks: for first-order Hamilton-Jacobi (HJ in short)
equations, i.e., when 
$A_\theta\equiv 0,$ and for uniformly parabolic equations.
It appears that there is a gap in the type of results and in
their proofs which are different.
\smallskip

In this work, we  investigate the situation in between.
We obtain a new proof for the large time behavior
of fully nonlinear degenerate second order equations which includes
most of the two previous type of results and allows to deal
with some everywhere degenerate second order equations.
According to our knowldege, the only result in this direction
is the one of Cagnetti et al.~\cite{cgmt13} where a
particular case of degenerate viscous HJ
equation is treated with a completely different
approach.
\smallskip

The precise  assumptions and the statements of our results 
are listed in the next section but let us 
describe the main ideas. We suppose that
there exists a, possibly empty, subset
\begin{eqnarray*}
&& \Sigma=\{x \in \T^N : A_\theta(x)=0 \text{ for all } \theta\in \Theta\}
\end{eqnarray*}
where the Hamiltonian $H_\theta(x,p)$ satisfies some first-order type assumptions
and the equation is uniformly parabolic outside $\Sigma,$ i.e.,
for all $\delta >0,$ there exists $\nu_\delta >0$ such that
\begin{eqnarray*}
&& \text{$A_\theta(x)=\sigma_\theta(x)\sigma_\theta(x)^T\geq \nu_\delta I$ 
 for $x\in \Sigma_\delta^C:=\{{\rm dist}(\cdot,\Sigma)> \delta\}$}.
\end{eqnarray*}
Actually, we are able to replace this assumption with a
weaker condition of ellipticity like
\begin{eqnarray}\label{hyp-simpl}
&& \left\{
\begin{array}{l}
\text{for all $\delta >0,$ there exists $\psi_\delta\in C^2(\T^N)$ such that}\\[1mm]
\displaystyle \mathop{\rm sup}_{\theta\in\Theta} \{-{\rm trace}(A_\theta(x)D^2 \psi_\delta)\}
-C|D\psi_\delta| >0 \quad \text{in $\Sigma_\delta^C:=\{{\rm dist}(\cdot,\Sigma)>\delta\}$} 
\end{array}\right.
\end{eqnarray}
(see \eqref{sur-sol-stricte} for the more general assumption).
It can be interpreted as follows.
When considering the exit time stochastic control problem associated 
with the equation in~\eqref{hyp-simpl}, it means that the controlled process
leaves $\Sigma_\delta^C$ almost surely in finite time.
\smallskip

Assuming that there exists a solution $(c,v)\in \R\times W^{1,\infty}(\T^N)$
of the ergodic problem associated with~\eqref{DHJE-intro}, namely
\begin{eqnarray*}
\mathop{\rm sup}_{\theta\in\Theta}
\{-{\rm trace}(A_\theta(x)D^2 v) + H_\theta(x,Dv)\}=c,
\quad x\in\T^N
\end{eqnarray*}
and that $\{u(\cdot, t)+ct, t\geq 0\}$ enjoys suitable compactness properties
in $W^{1,\infty}(\T^N),$
we obtain that
\begin{eqnarray}\label{asympt-beh}
u(x,t)+ct \to u_\infty(x) \quad \text{in $C(\T^N)$ as $t\to +\infty$}
\end{eqnarray}
in the two following frameworks.
\smallskip

The first case is when the $H_\theta$'s are strictly convex in $\Sigma,$
uniformly with respect to $\theta$ (see~\eqref{cvx_neuf} and Theorems~\ref{strict_co_dege}
and~\ref{cas-degenere}). 
A typical example,
which includes the mechanical Hamiltonian
$|p|^2+\ell(x),$ is
\begin{eqnarray}\label{hyp-superl-intro}
&& H_\theta(x,p)= a_\theta(x)|p|^{1+\alpha_\theta}+\langle b_\theta(x), p\rangle
+\ell_\theta(x),\\[2mm]
&& 1 < \underline{\alpha} \leq \alpha_\theta \leq \overline{\alpha},
\quad 0< \underline{a}\leq  a_\theta(x)\leq C,
\nonumber
\end{eqnarray}
and $a_\theta, b_\theta, \ell_\theta$ are bounded Lipschitz continuous
uniformly with respect to $\theta.$
Another example is the case of uniformly convex Hamiltonians
for which $(H_\theta)_{pp}(x,p)\geq 2\underline{a} I.$
\smallskip

The second case is, roughly speaking, when the $H_\theta$'s satisfy
\begin{eqnarray}\label{H10-intro}
&& \mathop{\rm inf}_{\theta\in\Theta}\{H_\theta(x,\mu p)-\mu H_\theta(x,p)\}
\geq (1-\mu )c \quad \text{for $\mu >1,$ $(x,p)\in\T^N\times\R^N,$}
\end{eqnarray}
with a strict inequality for $x\in\Sigma$ and $p\not= 0$
(see Assumption~\eqref{H10} and Theorems~\ref{mainresult} and~\ref{cas-degenere}).
This is also a convexity-like assumption close to the one introduced in 
Barles-Souganidis~\cite{bs00}
for first-order HJ equations. This assumption may appear to be restrictive 
in the sense that, in general, we do not know the exact value of the ergodic constant $c$
which appears in~\eqref{H10-intro}. The main motivations to deal with such a case
are, at first, it holds for some nonconvex cases (see Example~\ref{ex-non-cvx1}) 
which are a recurrent difficulty in HJ theory. Secondly, it allows to deal with
{\it Namah-Roquejoffre Hamiltonians} \cite{nr99} (see Section~\ref{subsec-nr})
$H(x,p)=F(x,p)-f(x),$
where $F$ is convex (but may be not strictly convex), $F(x,p)\geq F(x,0)=0.$
When the minimum of $f$ is achieved on $\Sigma,$ we can
calculate explicitely the ergodic constant, check that~\eqref{H10-intro} holds
and obtain the convergence~\eqref{asympt-beh}.
\smallskip

Detailled examples of applications are given in Section~\ref{sec:exples}
but let us give now a typical control-independent example.
Consider
\begin{eqnarray*}
&& \frac{\partial u}{\partial t}-a(x)^2 
\, {\rm trace}(\bar{\sigma}\bar{\sigma}^T D^2 u)+H(x,Du)=0,
\end{eqnarray*}
where $a\in W^{1,\infty}(\T^N),$ $\bar{\sigma}\in \mathcal{M}_N$ is a constant
matrix and $H$ is convex on $\T^N$ and strictly convex on  $\Sigma.$

\begin{itemize}

\item When $\bar{\sigma}$ is invertible then
the convergence~\eqref{asympt-beh} holds by 
Theorem~\ref{strict_co_dege} without further
assumptions on $a.$

\item When $a$ vanishes on $\partial [0,1]^N,$ then the 
convergence~\eqref{asympt-beh} holds by 
Theorem~\ref{cas-degenere} and Proposition~\ref{cor-degenere},
for any matrix $\bar{\sigma}$
(degenerate or not).

\end{itemize}

Let us recall the existing results and compare with ours.
The asymptotic behavior of~\eqref{DHJE-intro}
was extensively studied for
totally degenerate equations, i.e., first-order HJ
equations for which $\Sigma=\T^N,$ see Namah-Roquejoffre~\cite{nr99}, Fathi~\cite{fathi98}, 
Davini-Siconolfi~\cite{ds06}, Barles-Souganidis~\cite{bs00}, Barles-Ishii-Mitake~\cite{bim13}
(and the references therein for convergence results in bounded sets 
with various boundary conditions or in $\R^N$).
An assumption similar to~\eqref{H10-intro} was introduced 
in~\cite{bs00} to encompass all the previous works on first-order HJ equations
and to extend them to some nonconvex Hamiltonians. The arguments of~\cite{bs00} 
were recently revisited and simplified in
Barles-Ishii-Mitake~\cite{bim13}. 
Due to the above works, it is therefore natural to assume the strict convexity
of $H_\theta$ or~\eqref{H10-intro} on $\Sigma,$
which is the area where the equation is totally degenerate, and we recover
most of the previous results when taking $\Sigma=\T^N.$
\smallskip

As far as second order parabolic equations are concerned, there are less results
in the periodic setting.
Barles-Souganidis~\cite{bs01} obtained the asymptotic behavior~\eqref{asympt-beh}
in two contexts for
\begin{eqnarray*}
&& \frac{\partial u}{\partial t}-\Delta u + H(x,Du)=0\quad (x,t)\in\T^N\times (0,+\infty).
\end{eqnarray*}
The first one is when the Hamiltonian $H$ is {\em sublinear}, i.e.,
typically when $|H(x,p)|\leq C(1+|p|).$ The second one is
for {\em superlinear} Hamiltonians, i.e.,
typically when $H(x,p)$ is given by~\eqref{hyp-superl-intro}
(see~\eqref{BSsuperlinear} for a precise assumption).
Some extensions are given when $-\Delta u$ is replaced by
$-{\rm trace}(A(x,Du)D^2u)$ but the convergence result holds for
{\rm uniformly parabolic} equations. The reason is that the proof of convergence
is based on the strong maximum principle and, up to our knowledge, it is
the case for all results for second order equations except in the
recent work of Cagnetti et al.~\cite{cgmt13}. In this paper, the authors
obtained the convergence~\eqref{asympt-beh} for~\eqref{DHJE-intro} with
assumptions very close to ours in the particular case
of control-independent uniformly convex Hamiltonians (see Remark~\ref{rem-ref}
for details). Their approach is completely different and relies strongly
on the linearity with respect to $D^2u$ of the equation. 
We refer the reader to Tabet Tchamba~\cite{tabet10} and Fujita-Ishii-Loreti~\cite{fil06}
and the references therein for
related results of convergence for uniformly parabolic equations
in different settings (bounded sets, in $\R^N$).
\smallskip

The main step in the
proof of our results is the following.
We prove that that each $\tilde{u}$ in the $\omega$-limit set of $u+ct$ in
$C(\T^N\times [0,+\infty))$ is nonincreasing in $t$ thus 
$\tilde{u}(x,t)\to u_\infty(x)$ as $t\to +\infty.$
The convergence~\eqref{asympt-beh} then follows easily. 
To prove this main step, it is enough to show that
\begin{eqnarray}\label{P-intro}
&& \mathop{\rm sup}_{x\in\T^N} P_\eta[\tilde{u}](x,t),
\quad \text{with }
 P_\eta[\tilde{u}](x,t)=\mathop{\rm sup}_{s\geq t}\{\tilde{u}(x,t)-\tilde{u}(x,s)-\eta (s-t)\},
\end{eqnarray}
is a nonpositive constant ${m}_\eta$ for every $\eta >0.$
We argue by contradiction assuming ${m}_\eta >0.$ Since, by the stability result,
$\tilde{u}$ is still solution of~\eqref{DHJE-intro}, we obtain
that $P_\eta[\tilde{u}]$ is a subsolution
of a linearized
equation of the form
\begin{eqnarray*}
 &&\frac{\partial U}{\partial t}
+\mathop{\rm inf}_{\theta\in\Theta}\big\{
-{\rm tr}(\sigma_\theta(x)\sigma_\theta(x)^TD^2U)\big\}
-C|DU|\leq 0 \quad (x,t)\in\T^N\times (0,+\infty).
\end{eqnarray*}
In the set $\Sigma^C,$ we use the ellipticity-like condition~\eqref{hyp-simpl}
and strong maximum principle arguments to show that the maximum 
in~\eqref{P-intro} is achieved at $x\in \Sigma.$ 
In the set $\Sigma$
where $A_\theta =0,$ we have formally a first-order equation. We then apply
the first-order type assumptions, $H_\theta$
strictly convex or satisfying~\eqref{H10-intro}, to prove that ${m}_\eta$ 
cannot be positive. The main difficulty at this step is to control the 
second order terms  in~\eqref{DHJE-intro} near $\Sigma,$ see the proof 
of Lemma~\ref{IMpo1} for details.
\smallskip

The paper is organized as follows. In Section~\ref{sec:stat},
we start by introducing some steady assumptions for~\eqref{DHJE-intro}
which are in force in all the paper. We state 
Theorem~\ref{strict_co_dege} (strictly convex Hamiltonians) and 
Theorem~\ref{mainresult} (nonconvex cases) when~\eqref{DHJE-intro}
is uniformly parabolic outside $\Sigma$ since it is a more simpler
and natural case. Then we extend these results to a more
degenerate framework, see Theorem~\ref{cas-degenere}. 
Some concrete examples are gathered
in Section~\ref{sec:exples}. We also
introduce {\em superlinear Hamiltonians} for which all the steady
assumptions of Section~\ref{sec:basic-assumptions}
are satisfied. The rest of the paper
is devoted to the proofs. 
The strategy of proof is the same
for the three convergence results. It is why the core of the paper
is Section~\ref{strict_dege} where Theorem~\ref{strict_co_dege}
is proved. It relies on several lemmas. Section~\ref{mix_main} 
and  the last Section~\ref{sec:dege}
are devoted, respectively,
to the proofs of Theorem~\ref{mainresult}
and Theorem~\ref{cas-degenere} and their applications.
\smallskip

\noindent
{\bf Acknowledgement.} We would like to thank Guy Barles
for bringing to our knowledge the paper~\cite{bim13} which allowed us
to simplify our proofs.
This work was partially supported by the ANR (Agence Nationale de
la Recherche) through HJnet project ANR-12-BS01-0008-01
and WKBHJ project ANR-12-BS01-0020.

\section{Statement of the results}
\label{sec:stat}

\subsection{Setting of the problem and first assumptions}
\label{sec:basic-assumptions}

We consider
\begin{eqnarray}\label{DHJE}
&& \left\{
\begin{array}{ll}
\displaystyle\frac{\partial u}{\partial t}
+\mathop{\rm sup}_{\theta\in\Theta}
\{-{\rm trace}(A_\theta(x)D^2 u) + H_\theta(x,Du)\}=0, &
(x,t)\in \T^N\times (0,+\infty),\\[3mm]
u(x,0)=u_{0}(x), & x\in\T^N,
\end{array}
\right.
\end{eqnarray}
and, for $\lambda>0,$ the associate approximate stationary equation
\begin{eqnarray}\label{DHJE-sta}
\lambda v_\lambda+\mathop{\rm sup}_{\theta\in\Theta}
\{-{\rm trace}(A_\theta(x)D^2 v_\lambda) + H_\theta(x,Dv_\lambda)\}=0, \quad
x\in \T^N.
\end{eqnarray}

The following assumptions will be in force in all the paper.
The set $\Theta$ is a metric space. Let $C>0$ be a fixed constant 
(independent of $\theta$).
\begin{eqnarray}\label{H1_dege}
&& 
\text{For all $\theta\in\Theta,$  $A_\theta=\sigma_\theta \sigma_\theta^T,$
$\sigma_\theta\in W^{1,\infty}(\T^N;\mathcal{M}_N)$ 
with $|\sigma_\theta|,|D\sigma_\theta|\leq C$;}
\end{eqnarray}
\begin{eqnarray}\label{Hunif-cont}
&& \left\{\begin{array}{l}
\text{for all $\theta\in\Theta,$ 
${H}_\theta\in W^{1,\infty}_{\rm loc}(\T^N\times \R^N),$ 
$|{H}_\theta(x,0)|\leq C,$}\\
\text{for all $R\!>\!0,$ there exists $C_R>0$ independent of $\theta$ such that}\\
\quad|{H}_\theta(x,p)-{H}_\theta(y,q)|\leq C_R(|x-y|+|p-q|), \ x,y\in\T^N, 
|p|,|q|\leq R.
\end{array}\right.
\end{eqnarray}
These assumptions are natural when dealing with Hamilton-Jacobi equations.
Notice that~\eqref{Hunif-cont} is automatically satisfied when there is no control.
Moreover, we assume
\begin{eqnarray}\label{sol-lip-hold}
&& \left\{\begin{array}{l}
\text{There exists viscosity solutions $u\in C(\T^N\times [0,+\infty))$
and $v_\lambda\in C(\T^N)$}\\
\text{of~\eqref{DHJE} and~\eqref{DHJE-sta} respectively with}\\
\text{$|u(x,t)-u(y,t)|, |v_\lambda(x)-v_\lambda(y)| \leq C|x-y|,$ \ $x,y\in\T^N, t\geq 0, \lambda >0.$}
\end{array}\right.
\end{eqnarray}
Besides the existence of a continuous viscosity solution of the
equation, we assume gradient bounds {\em independent}
of $t$ and $\lambda.$ This is a crucial point and the first
step when trying to prove asymptotic results.
Let us give some important consequences of~\eqref{sol-lip-hold}.
At first, we have
a comparison principle for~\eqref{DHJE} and~\eqref{DHJE-sta}.
By the comparison principle (for instance for~\eqref{DHJE}), 
we mean that, if $u_1$ and $u_2$ are respectively USC subsolution
and LSC supersolution of~\eqref{DHJE} and either $u_1$
or $u_2$ satisfies the Lipschitz continuity of~\eqref{sol-lip-hold}
then $u_1-u_2\leq {\rm sup}_{\T^N} \{(u_1-u_2)^+ (\cdot,0)\}.$ 
In particular, we have uniqueness of the solutions
of~\eqref{DHJE}-\eqref{DHJE-sta}
in the class of functions satisfying the Lipschitz continuity 
of~\eqref{sol-lip-hold}.
The second consequence is that we can solve the ergodic problem
associated with~\eqref{DHJE}.
More precisely, there exists a unique $c\in\R$
and $v\in W^{1,\infty}(\T^N)$ solutions of
\begin{eqnarray}\label{sta}
\mathop{\rm sup}_{\theta\in\Theta}
\{-{\rm trace}(A_\theta(x)D^2 v) + H_\theta(x,Dv)\}=c, \quad
x\in \T^N.
\end{eqnarray}
A byproduct is $|u(x,t)+ct|\leq C.$ 
The proofs of these results are classical (see for instance~\cite{bs01, ln13a})
so we skip them. In Section~\ref{sec:exples}, we introduce 
{\em superlinear Hamiltonians} for which
the above assumptions are satisfied.
Since the above basic assumptions will be used in all our results, for
shortness, we introduce a steady assumption collecting them
\begin{eqnarray}\label{steady}
&&\text{Assumptions~\eqref{H1_dege}, \eqref{Hunif-cont}, \eqref{sol-lip-hold} hold.}
\end{eqnarray}

We recall that
\begin{eqnarray}\label{defF}
&&\Sigma=\{x \in \T^N : A_\theta(x)=\sigma_\theta(x)\sigma_\theta^T(x)=0 
\ {\rm for \ all \ }\theta\in\Theta\},
\end{eqnarray}
and, for the two first convergence results which follow, we assume a
nondegeneracy assumption for $\sigma_\theta$ holds outside $\Sigma$:
\begin{eqnarray}\label{inver_dege}
&& \left\{\begin{array}{l}
\text{for all $\delta >0,$ there exists $\nu_\delta >0$ 
such that for all $\theta\in\Theta$}\\ 
\text{$A_\theta(x)=\sigma_\theta(x)\sigma_\theta(x)^T\geq \nu_\delta I$ 
 for $x\in \Sigma_\delta^C:=\{{\rm dist}(\cdot,\Sigma)> \delta\}$}.
\end{array}
\right.
\end{eqnarray}
This assumption is replaced by a weaker one in Section~\ref{sec:thm-dege}.

\subsection{A convergence result for strictly convex Hamiltonians}

The main assumption in this section is
\begin{eqnarray}\label{cvx_neuf}
&& 
\left\{\begin{array}{l}
\text{For all $x\in\Sigma,$ 
$0<\lambda <1$ and $p,q\in\R^N$ such that $p\not=q$,}\\[2mm]
\displaystyle\inf_{\theta\in\Theta} \{
\lambda H_{\theta}(x,p)+(1-\lambda) H_{\theta}(x,q)-H_{\theta}(x,\lambda p+(1-\lambda)q)\}> 0.
\end{array}\right.
\end{eqnarray}
This condition is a strict convexity assumption on the $H_\theta$'s on $\Sigma$
uniformly with respect to $\theta.$

%

\begin{thm}\label{strict_co_dege}
Suppose \eqref{steady},
\eqref{inver_dege}, \eqref{cvx_neuf} hold
and that $H_\theta (x,\cdot)$ is convex for every $x\in\T^N.$
Then $u(x,t)+ct \to u_\infty(x)$ in $C(\T^N)$  when $t\to +\infty,$ 
where $u$ is the solution of~\eqref{DHJE}
and $u_\infty$ is a solution of \eqref{sta}.
\end{thm}

Section~\ref{strict_dege} is devoted to the proof.

\subsection{A convergence result for non necessarily convex 
Hamiltonians}

We will assume the following
for the Hamiltonians $H_\theta$'s. Recall that $c$ denotes
the ergodic constant in~\eqref{sta}. There exists $\mu_0>1$
such that
\begin{eqnarray}\label{H10}
&& \left\{
\begin{array}{ll}
\text{(i) $H_\theta(x,\mu p)-\mu H_\theta(x,p)\geq (1-\mu )c$
for all $(x,p)\in \T^N\times \R^N$, $1<\mu< \mu_0,$}\\[2mm]
\text{(ii) There exists a, possibly empty, compact set $K$ of $\Sigma$ such that}\\
\text{\quad(a) $H_\theta(x,p) \ge c$ for all  $(x,p)\in K \times \R^{N},$}\\
\text{\quad(b) for all $x\in\Sigma,$ $p\in\R^N,$ $1<\mu\leq \mu_0,$
if $d(x,K)\not= 0,$ $p\not= 0$,  then}\\ 
\hspace*{1.2cm}\text{$\displaystyle
\mathop{\rm inf}_{\theta\in\Theta}\{H_\theta(x,\mu p)-\mu H_\theta(x,p)\}> (1-\mu )c$.}
\end{array}
\right.
\end{eqnarray}

\begin{thm}\label{mainresult}
Suppose that \eqref{steady}, 
\eqref{inver_dege} and \eqref{H10} hold.
Then $u(x,t)+ct \to u_\infty(x)$ in $C(\T^N)$  when $t\to +\infty,$ 
where $u$ is the solution of~\eqref{DHJE}
and $u_\infty$ is a solution of \eqref{sta}.
\end{thm}

The proof of this theorem is done in Section~\ref{mix_main}.

We make some comments about the assumptions. 
Conditions \eqref{H10}(i) and \eqref{H10}(ii)(b) are some kind of convexity
requirements but it may apply to some nonconvex Hamiltonians (see Section~\ref{sec:exples}).
Taking, $p=0$ in~\eqref{H10}(i), we obtain 
\begin{eqnarray}\label{0sous-sol}
H_\theta(x,0)\leq c, \qquad x\in\T^N, \theta\in\Theta, 
\end{eqnarray}
which implies
that $v\equiv 0$ is a subsolution of~\eqref{sta}.

Assumption~\eqref{H10} may be seen restrictive.
Indeed, in general one does not know the exact value of the ergodic
constant $c$ so it is difficult to check that $\eqref{H10}$ holds.
We have three motivations to state such a result. At first, 
there are some interesting cases for which we can calculate
the exact value of $c$ and~\eqref{H10} holds (see Proposition~\ref{valc-nr}).
It allows to treat some {\em Namah-Roquejoffre type} Hamiltonians, see
Section~\ref{subsec-nr}.
Secondly, this assumption
encompasses nonconvex Hamiltonians (see Section~\ref{sec:exples})
and such nonconvex cases are hard to deal with. 
Finally, it is worth pointing out that,
when there exist $C^2$ subsolutions of~\eqref{sta},
then Theorem~\ref{strict_co_dege} appears as an immediate corollary of 
Theorem~\ref{mainresult} (see Remark~\ref{conv-sous-solC2}).

\begin{prop}\label{valc-nr}
Assume~\eqref{steady} and
\begin{eqnarray}\label{hypHF}
&& 
H_\theta(x,p)\geq H_\theta(x,0) \text{ for $(x,p)\in\Sigma\times\R^N,$
$\theta\in\Theta.$}
\end{eqnarray}
If, in addition, either
\begin{eqnarray}\label{supTS}
\displaystyle \mathop{\rm sup}_{x\in\T^N, \theta\in\Theta} H_\theta(x,0)
=  \mathop{\rm sup}_{x\in\Sigma, \theta\in\Theta} H_\theta(x,0)
\end{eqnarray}
holds or~\eqref{inver_dege} and
\begin{eqnarray}\label{Hpresque-cvx}
&& H_\theta(x,\mu p)-\mu H_\theta(x,p)\geq (1-\mu )H_\theta(x,0)
\text{ for $(x,p)\in\T^N\times\R^N,$
$\theta\in\Theta,$ $\mu >1,$}
\end{eqnarray}
hold, then
$\displaystyle c= \mathop{\rm sup}_{x\in\Sigma, \theta\in\Theta} H_\theta(x,0).$
\end{prop}

This proposition, the proof of which is given in Section~\ref{mix_main},
is used to apply Theorem~\ref{mainresult}
for Hamiltonians of {\rm Namah-Roquejoffre type} in Section~\ref{sec:exples}.
We see that the value of the ergodic constant is affected by the second-order terms
in the sense that it is not the same as for~\eqref{DHJE} with $\sigma_\theta\equiv 0.$ 
Assumption~\eqref{supTS} requires that the supremum of $H_\theta(\cdot, 0)$ is actually
achieved where the diffusion vanishes. Assumption~\eqref{Hpresque-cvx} holds automatically
when $H_\theta$ is convex.

\begin{rem}\label{conv-sous-solC2}
We sketch the proof of the fact that,
if there exists a $C^2$ subsolution of~\eqref{sta},
then Theorem~\ref{strict_co_dege} is a corollary of
Theorem~\ref{mainresult}.
Assuming that $u$ is the solution of \eqref{DHJE} under the assumptions
of~Theorem~\ref{strict_co_dege}
 and $v$ is a $C^2$ subsolution  of \eqref{sta},
we set $w=u+ct-v.$ Then $w$ is the bounded solution of
\begin{eqnarray*}
\displaystyle\frac{\partial w}{\partial t}-{\rm trace}(A(x)D^2w)+ H(x, Dv+Dw )-H(x, Dv)-g(x)=0,
\end{eqnarray*} 
where $g(x):={\rm trace}(A(x)D^2 v)-H(x, Dv )+c \ge 0$ is continuous
since $v$ is {\em $C^2.$}
Introducing the new Hamiltonian $G(x,p)=H(x, p+Dv )-H(x, Dv)-g(x),$
it is not difficult to check that the strict convexity assumption~\eqref{cvx_neuf}
for $H$ implies that
$G$ satisfies~\eqref{H10} with $c=0$ and
$K=\emptyset$ (for $p$ bounded which is enough since $u$ and $v$ are
Lipschitz continuous in $x$).  
We then apply Theorem~\ref{mainresult} to the new equation
to  obtain the large time behavior of $u.$
Actually, it is possible to generalize such a proof
when there exists a {\em $C^{1,1}$} subsolution of~\eqref{sta}
but the proof is much more involved.
We mention this slight extension because
it is known (Bernard~\cite{bernard07b}) that there exists {\em $C^{1,1}$} subsolutions
of~\eqref{sta} for first order HJ (i.e., when $\Sigma=\T^N$)
under general assumptions.
\end{rem}

\subsection{A more general result of convergence}
\label{sec:thm-dege}

We now generalize the two previous results when~\eqref{inver_dege}
is replaced by a weaker assumption. The proof of the results
of this section are given in Section~\ref{sec:dege}.

Before stating our main assumption, let us introduce
some notations. We denote by $\pi : \R^N\to \T^N$
the canonical projection and we add a superscript $\sim$ to
the coset representatives of the objects defined on $\T^N.$
For instance, $\tilde{\Sigma}$ is a 1-periodic subset of $\R^N$
such that $\pi (\tilde{\Sigma})=\Sigma$ and 
$\tilde{\sigma}_\theta(\tilde{x})=\sigma_\theta(\pi(\tilde{x}))$
for any $\tilde{x}\in\R^N.$

We assume, for some $C>0,$
\begin{eqnarray}\label{sur-sol-stricte}
&& \left\{\begin{array}{l}
\text{For all $\delta >0,$ 
there exists $\tilde{\psi}_\delta\in C^2(\R^N)$ and an open set
$\Omega_\delta \subset\subset\R^N$
such that}\\
\text{$[0,1]^N \subset \Omega_\delta,$} \  \
\text{ $\tilde{\psi}_\delta \leq 0$ in $\Omega_\delta$,}\  \
\text{ $\tilde{\psi}_\delta \geq 0$ in $\Omega_\delta^C$,}\\[2mm]
\displaystyle\mathop{\rm inf}_{\theta\in\Theta}
\{-{\rm trace}(\tilde{A}_\theta(\tilde{x})D^2 \tilde{\psi}_\delta(\tilde{x}))\} 
-C|D\tilde{\psi}_\delta(\tilde{x})| >0
\ \text{ for } \tilde{x}\in \tilde{\Sigma}_\delta^C\cap \Omega_\delta,\\[2mm]
\text{ where $ \tilde{\Sigma}_\delta=\{\tilde{x}\in\R^N : 
{\rm dist}(\tilde{x}, \tilde{\Sigma})\leq \delta\}.$}
\end{array}
\right.
\end{eqnarray}

\begin{thm}\label{cas-degenere}
We assume that either the assumptions of Theorem~\ref{mainresult}
or the assumptions of Theorem~\ref{strict_co_dege} hold,
where~\eqref{inver_dege} is replaced by~\eqref{sur-sol-stricte}
in both cases.
Then $u(x,t)+ct$ converges uniformly to a solution 
$v(x)$ of \eqref{sta} when $t\to +\infty.$ 
\end{thm}

The difference with the previous theorems is that we do not assume
the uniform ellipticity assumption~\eqref{inver_dege}.
We consider the weaker assumption~\eqref{sur-sol-stricte} instead
(see Proposition~\ref{ell-sur-sol}). 
This latter assumption
allows to deal with some fully nonlinear everywhere degenerate equations.
It is written in a tedious
way since, in some cases, we need to construct a supersolution which is not
1-periodic (and therefore it is not a function on $\T^N$).

\begin{prop}\label{ell-sur-sol}
If~$\Sigma\not=\emptyset$ and \eqref{inver_dege} holds, 
then~\eqref{sur-sol-stricte} holds. 
\end{prop}

It follows that Theorems~\ref{strict_co_dege} and~\ref{mainresult} are
corollary of Theorem~\ref{cas-degenere} when $\Sigma\not=\emptyset$. 
We can apply the theorem to obtain the convergence for some everywhere 
degenerate equations. Let us give an application.

\begin{prop}\label{cor-degenere}
Assume \eqref{H1_dege}, \eqref{Hunif-cont},
\begin{eqnarray}\label{cont-x-theta}
&& (x,\theta)\in \T^N\times \Theta \mapsto \sigma_\theta (x)
\text{ is continuous}, \quad \text{$\Theta$ is compact,}
\end{eqnarray}
\begin{eqnarray}\label{F-bord-tore}
&& \bigcup_{1\leq i\leq N}\{x=(x_1,\cdots, x_N)\in \T^N: x_i=0\} \subset \Sigma
\end{eqnarray}
and
\begin{eqnarray}\label{non-deg-sigma}
&& \bigcup_{x\in \Sigma^C, \theta\in\Theta} {\rm ker}(\sigma_\theta(x))\cap \mathbb{S}^{N-1}
\not= \mathbb{S}^{N-1}:=\{x\in\R^N : |x|=1\}.
\end{eqnarray}
Then Assumption~\eqref{sur-sol-stricte} holds.
\end{prop}

The assumption~\eqref{F-bord-tore} means that the boundary of the cube $[0,1]^N$
is contained in $\Sigma$; more generally, we need the connected components
of $\tilde{\Sigma}^C$ to be bounded in $\R^N.$ 
In  $\Sigma^C,$  ${\rm ker}(\sigma_\theta(x))$
is at most an hyperplane. The assumption~\eqref{non-deg-sigma} means that the
union of these hyperplanes does not fulfill the whole space. 
Some concrete examples of applications are given in Section~\ref{sec:exples}.

\section{Applications and examples}
\label{sec:exples}

\subsection{Superlinear Hamiltonians}
We first introduce an assumption on the $H_\theta$'s, called {\em superlinear}
in~\cite{bs01}, under which the steady assumptions of Section~\ref{sec:basic-assumptions}
hold.
\begin{eqnarray}\label{BSsuperlinear}
&&  \left\{\begin{array}{l}
\text{There exists } L_1\geq 1 \text{ such that if $|p|\geq L_1,$ then}\\
\displaystyle  \quad
L_1\left[({H}_\theta)_{p}p-\!{H}_\theta\!
-\!|2\sigma_\theta (\sigma^T_\theta)_x| | p|
-|{H}_\theta(\cdot,0)|_\infty\right]\\
\hspace*{1.5cm}
-|({H}_\theta)_{x}| -N|(\sigma_\theta)_x|_\infty^2|p|  \} \geq 0,
\text{ for a.e. } (x,p)\in \T^N \times \R^N,  \theta\in\Theta.
\end{array}\right.
\end{eqnarray}

\begin{thm}\label{bornes-sur-lin}
Assume \eqref{H1_dege}, \eqref{Hunif-cont} and~\eqref{BSsuperlinear}.
Then~\eqref{sol-lip-hold} holds, we have a comparison principle
for~\eqref{DHJE} and~\eqref{DHJE-sta}. In particular, the ergodic problem~\eqref{sta}
has a solution $(c,v)\in \R\times W^{1,\infty}(\T^N).$ 
\end{thm}

The main ingredients in the proof of this result are 
gradient bounds for the solutions of~\eqref{DHJE} and~\eqref{DHJE-sta}
uniform in $t$ and $\lambda$ respectively.
We refer the reader to Barles-Souganidis~\cite{bs01}
and~\cite{ln13a}.

We give some examples of Hamiltonians satisfying both
the assumption of Theorem~\ref{bornes-sur-lin} and~\eqref{cvx_neuf}.

\begin{ex} \label{ex-strict-cvx}
(strictly convex Hamilton-Jacobi-Bellman equations)
We suppose that~\eqref{H1_dege} holds and
\begin{eqnarray*}
&& H_\theta(x,p)= a_\theta(x)|p|^{1+\alpha_\theta}+\langle b_\theta(x), p\rangle
+\ell_\theta(x),\\[2mm]
&& 1 < \underline{\alpha} \leq \alpha_\theta \leq \overline{\alpha},
\quad 0< \underline{a}\leq  a_\theta(x)\leq C,\\
&&  |a_\theta(x)|,|b_\theta(x)|,|\ell_\theta(x)|,|a_\theta(x)\!-\!a_\theta(y)|,
|b_\theta(x)\!-\!b_\theta(y)|,|\ell_\theta(x)\!-\!\ell_\theta(y)|\leq C, \ x,y\in\T^N.
\end{eqnarray*}
\end{ex}

\begin{ex} \label{ex-unif-cvx}
(uniformly convex Hamilton-Jacobi-Bellman equations)
We suppose that~\eqref{H1_dege}-\eqref{Hunif-cont} hold and
\begin{eqnarray*}
&& (H_\theta)_{pp}(x,p)\geq \underline{h}>0,
\qquad |(H_\theta)_x|\leq C(1+|p|^2).
\end{eqnarray*}
\end{ex}

We now give an example 
such that the assumptions of Theorem~\ref{bornes-sur-lin}
still holds but the Hamiltonian is not convex anymore
and satisfies~\eqref{H10}.

\begin{ex} \label{ex-non-cvx1} (nonconvex equations)
We adapt an example from~\cite{bs00}. We consider~\eqref{DHJE}
without control with
\begin{eqnarray}\label{ex1_cha_dege}
H(x,p)=\psi(x,p)F(x,\frac{p}{|p|})-f(x),
\end{eqnarray}
where $f \in W^{1,\infty}(\T^N)$ is nonnegative, 
$F\in W^{1,\infty}(\T^N \times \R^N)$ is
strictly positive and $\psi(x,p)=|p+h(x)|^2-|h(x)|^2,$
with $h\in W^{1,\infty}(\T^N;\R^N).$ Notice that $H\in W^{1,\infty}(\T^N\times\R^N)$
is not convex in general.
We suppose that
$A=\sigma\sigma^T$ satisfies~\eqref{H1_dege} and
\begin{eqnarray}\label{uni_dege}
\{x\in \T^N : f(x)=|h(x)|=|\sigma(x)|=0\} \neq \emptyset.
\end{eqnarray} 
Arguing as in the proof of Proposition~\ref{valc-nr}, we can show that
$c=0$ (we cannot applying Proposition~\ref{valc-nr} directly since~\eqref{hypHF}
does not hold).

We now prove that $H$ satisfies~\eqref{H10} with
$K=\emptyset.$ For every $\mu>1$,
we have 
\begin{eqnarray*}
H(x,\mu p)-\mu H(x,p)=(\mu^2-\mu)|p|^2F(x,\frac{p}{|p|})+(\mu-1)f(x) \ge 0,
\end{eqnarray*} 
so~\eqref{H10}(i) holds. If $p\not= 0,$ the above inequality is strict
and therefore~\eqref{H10}(ii)(b) holds.
\end{ex}

\subsection{Second-order equations satisfying~\eqref{inver_dege}
or~\eqref{sur-sol-stricte}}

\begin{ex}\label{inversible-no-control}
Without control, we choose $\sigma \in W^{1,\infty}(\T^N;\mathcal{M}_n)$
with for each $x\in\T^N,$ either $\sigma(x)=0$ or $\sigma(x)$ is invertible.
Then~\eqref{H1_dege} and \eqref{inver_dege} hold.
A particular case is $\sigma(x)=a(x)\bar{\sigma}$
where $a\in W^{1,\infty}(\T^N)$ and $\bar{\sigma}\in  \mathcal{M}_n$ is a
constant invertible matrix.
\end{ex}

\begin{ex}\label{inversible-avec-control}
With control, we can deal with some cases of fully nonlinear
equations. For instance, consider $\sigma_\theta(x)=a(x)\bar{\sigma}_\theta,$
where $a\in W^{1,\infty}(\T^N)$ as in Example~\ref{inversible-no-control}
and there exists $\nu>0$ such that, for all $\theta\in\Theta,$
$\bar{\sigma}_\theta\in  \mathcal{M}_n$ and $\bar{\sigma}_\theta\bar{\sigma}_\theta^T\geq \nu I.$
Then~\eqref{H1_dege} and \eqref{inver_dege} hold. 
\end{ex}

It is worth noticing that, in the two examples above, we may
have the two following particular cases: $a>0$ on $\T^N$ and the
equation is uniformly parabolic, or $a=0$ on  $\T^N$  and the
equation is a first-order HJ equation.

We now give some examples for which~\eqref{F-bord-tore}-\eqref{non-deg-sigma}
hold (and so~\eqref{sur-sol-stricte} holds thanks to Proposition~\ref{cor-degenere}).

\begin{ex}\label{cas-dege1}
A first control-independent case is $\sigma(x)=a(x)\bar{\sigma}$
where $a\in W^{1,\infty}(\T^N)$ is such that $\partial[0,1]^N\subset \{a=0\}$
and $\bar{\sigma}\in  \mathcal{M}_n$ is any  nonzero constant degenerate matrix. 
With control, we can take  $\sigma_\theta(x)=a(x)\bar{\sigma}_\theta$
where $a$ satisfies the same assumptions in the control-independent case
and $(\bar{\sigma}_\theta)_{\theta\in \Theta}$ is a finite set of 
 nonzero constant degenerate (in this case $\Theta=\{1,2,\cdots, k\}).$
In both case, \eqref{F-bord-tore}-\eqref{non-deg-sigma} holds.
\end{ex}

\begin{ex}\label{cas-dege2}
For simplicity, we consider a control-independent example.
Assume that $\sigma=(\sigma_{ij})_{1\leq i,j\leq N}
\in W^{1,\infty}(\T^N;\mathcal{M}_N)$
is such that
\begin{eqnarray*}
\partial[0,1]^N\subset\{\sigma_{11}=0\}\subset \Sigma=\{\sigma =0\}.
\end{eqnarray*}
Then, for all $x\in \Sigma^C,$ $\sigma(x)e_1\not= 0,$
where $e_1=(1,0,\cdots 0).$ Therefore~\eqref{non-deg-sigma} holds.
\end{ex}

\subsection{Application to convergence results}

In the following cases, there exist solutions
to~\eqref{DHJE} and~\eqref{sta} (i.e.,~\eqref{sol-lip-hold}
holds) and we have a convergence result:

\begin{itemize}

\item Applying Theorem~\ref{strict_co_dege} when the $H_\theta$'s are given 
by Examples~\ref{ex-strict-cvx}
or~\ref{ex-unif-cvx} and 
the diffusion matrices are given by Examples~\ref{inversible-no-control}
or~\ref{inversible-avec-control}.

\item Applying Theorem~\ref{mainresult} when the $H_\theta$'s
given by Example~\ref{ex-non-cvx1} and the diffusion matrices are given by 
Examples~\ref{inversible-no-control} or~\ref{inversible-avec-control}.

\item  Applying Theorem~\ref{cas-degenere} when
the $H_\theta$'s are given by 
Examples~\ref{ex-strict-cvx},~\ref{ex-unif-cvx}
or~\ref{ex-non-cvx1} and  the diffusion matrices are given by 
Examples~\ref{cas-dege1} or~\ref{cas-dege1}.

\end{itemize}

\begin{rem}\label{rem-ref}
When $\Sigma=\T^N$ or $\Sigma=\emptyset,$
these convergence results were obtained in~\cite{fathi98, nr99, bs00, ds06}
and~\cite{bs01} respectively.
In the particular case of control-independent $C^2$ uniformly convex Hamiltonian
(see Example~\ref{ex-unif-cvx}) with $\sigma(x)=a(x)I$ with $a\in C^2(\T^N)$
(see Example~\ref{inversible-no-control}), the result
is proven in~\cite{cgmt13} by using a nonlinear adjoint
method. Notice that, on the one side, we can deal with fully nonlinear equations
and, on the other side, we only require the Hamiltonians to be uniformly convex
on $\Sigma.$
\end{rem}

When the assumption~\eqref{BSsuperlinear} does not hold, we need to prove
a priori the existence of Lipschitz solutions
to~\eqref{DHJE} and~\eqref{sta} before applying a convergence result.
For instance, if~\eqref{sol-lip-hold} holds for~\eqref{eqstrict-cvxS}
in Example~\ref{ex-strict-cvxS} below, then we have the convergence
by applying Theorem~\ref{strict_co_dege}.
An other important case is given in Section~\ref{subsec-nr}.

\begin{ex} \label{ex-strict-cvxS}
Consider
\begin{eqnarray}\label{eqstrict-cvxS}
\frac{\partial u}{\partial t}- a(x)^2\Delta u + (1-a(x))|Du|^2 =f(x), \quad
(x,t)\in\T^N\times (0,+\infty),
\end{eqnarray}
where $a, f\in W^{1,\infty}(\T)$ and $a$ is defined by
\begin{eqnarray*}
a(x)=
&& \left\{
\begin{array}{ll}
1-|x| & \text{if $x\in [0,\frac{1}{4}],$}\\[2mm]
0  & \text{if $x\in [\frac{1}{4},\frac{3}{4}],$}\\[2mm]
|x|-1 & \text{if $x\in [\frac{3}{4},1].$}
\end{array}
\right.
\end{eqnarray*}
Then $H(x,p)=(1-a(x))|p|^2$ is striclty convex 
on $\Sigma=[\frac{1}{4},\frac{3}{4}]$
and~\eqref{cvx_neuf} holds. 
\end{ex}

We end this section by a counter-example.

\begin{ex}\label{contre-exple-vinh}
Consider
\begin{equation}\label{contre_exe1}
\left\{
  \begin{array}{ll}
\frac{\partial u}{\partial t}-{\rm trace}(A(x)D^2u)+ H(Du)=0,
   & (x,t)\in \T^2\times (0,+\infty),\\[5pt]
 u(x,0)={\rm sin}(x_1+x_2), & x=(x_1,x_2)\in\T^2,
  \end{array}
\right.
\end{equation}
where $A= a(x)^2 \scriptsize
\left(
\begin{array}{cc}
1 & -1 \\
-1 & 1
\end{array}
\right)
$
and $H(p)=\frac{1}{\sqrt{2}}|p+(1,1)|-1.$
The solution of~\eqref{contre_exe1} is
$u(x,t)={\rm sin}(x_1+x_2-t)$ and convergence fails as $t\to +\infty.$
In this example, $A(x)$ is degenerate and $H$ is convex but
does not satisfy neither~\eqref{cvx_neuf}
nor~\eqref{H10}.
\end{ex}

\subsection{The Namah-Roquejoffre case}\label{subsec-nr}
Consider
\begin{eqnarray}\label{HJ-nr}
&&\frac{\partial u}{\partial t}
+\mathop{\rm sup}_{\theta\in\Theta}
\{-{\rm trace}(A_\theta(x)D^2 u) + F_\theta(x,Du)-f_\theta(x)\}=0, \quad
(x,t)\in \T^N\times (0,+\infty),
\end{eqnarray}
where $A_\theta=\sigma_\theta\sigma_\theta^T$ satisfies~\eqref{H1_dege},
\begin{eqnarray}\label{hyp-nr123}
&&
\left\{\begin{array}{l}
\displaystyle f_\theta\in W^{1,\infty}(\T^N),
\quad \text{$F_\theta\in W_{\rm loc}^{1,\infty}(\T^N\times\R^N)$ is convex in $p,$}\\
\displaystyle F_\theta(x,p)\geq F_\theta(x,0)=0,
\\
\displaystyle 
K:=\{x\in\Sigma : f_\theta(x)=\mathop{\rm min}_{y\in\T^N,\theta\in\Theta}f_\theta(y)\}\not=\emptyset
\end{array}\right.
\end{eqnarray}
and
\begin{eqnarray}\label{hyp-nr456}
\text{for $x\in \Sigma\setminus K,$} \ \  \mathop{\rm inf}_{\theta\in\Theta} f_\theta(x) 
>\mathop{\rm inf}_{y\in\T^N, \theta\in\Theta} f_\theta(y).
\end{eqnarray}
We call such kind of Hamiltonians of Namah-Roquejoffre type,
see  \cite{nr99, bs00}.

When $F$ is strictly convex in $p,$ then the convergence
result for~\eqref{HJ-nr} can be obtained with the use of Theorem~\ref{strict_co_dege}.
Here, we want to deal with the typical Hamiltonian
which appears in~\cite{nr99}, that is,
$F_\theta(x,p)=a_\theta(x)|p|,$ which is not strictly convex 
and does not satisfy~\eqref{BSsuperlinear}.
It is why we assume here a priori that~\eqref{sol-lip-hold}
holds for~\eqref{HJ-nr}.

From Proposition~\ref{valc-nr}, we obtain 
$\displaystyle c= - \mathop{\rm min}_{x\in\Sigma, \theta\in\Theta} f_\theta (x).$
Therefore, for $x\in K,$ we have $H_\theta (x,0)=-f_\theta(x)=c$ and~\eqref{H10}(ii)(a)
holds.
By~\eqref{hyp-nr123}, we have, for all $x\in\T^N,$ $\mu >1,$
\begin{eqnarray*}
H_\theta(x,\mu p)-\mu H_\theta(x,p) 
&=&F_\theta(x,\mu p)-\mu F_\theta(x,p)-(1-\mu)f_\theta(x)\\
&\geq& -(1-\mu) \mathop{\rm min}_{\T^N,\Theta} f_\theta
=  -(1-\mu) \mathop{\rm min}_{\Sigma,\Theta} f_\theta.
\end{eqnarray*}
Therefore~\eqref{H10}(i) holds. By~\eqref{hyp-nr456},~\eqref{H10}(ii)(b) holds.

Therefore, assuming~\eqref{sol-lip-hold} for~\eqref{HJ-nr}
and~\eqref{hyp-nr123},~\eqref{hyp-nr456}, then
we obtain the convergence
from Theorem~\ref{mainresult} when $A_\theta$ satisfies~\eqref{inver_dege}
and from Theorem~\ref{cas-degenere} when $A_\theta$ satisfies
the conditions of Examples~\ref{cas-dege1} or~\ref{cas-dege2}.

\section{Proof of Theorem \ref{strict_co_dege}}
\label{strict_dege}

At first, we notice that we can assume without loss of generality
that $c=0$ in~\eqref{sta}. Indeed, by a change of function
$u(x,t)\to u(x,t)+ct,$ the new function satisfies~\eqref{DHJE}
where $H_\theta$ is replaced with $H_\theta-c$ and,
if $H_\theta$ satisfies the
strict convexity assumption~\eqref{cvx_neuf}, then $H_\theta-c$
still satisfies~\eqref{cvx_neuf}. So, we suppose that $c=0$
and the solution $u(x,t)$ of~\eqref{DHJE} is bounded.
We aim at proving that $u(x,t)$ converges
uniformly to some function $u_\infty(x),$ which is a
solution of~\eqref{sta} with $c=0$ by the stability result.
In the following, $v$ is a Lipschitz
continuous solution of~\eqref{sta} with $c=0.$ 

Following the ideas of~\cite{bim13, nguyen12},
for $\eta>0, \mu >1$ and $(x,t)\in \T^N\times (0,+\infty),$ we introduce
\begin{eqnarray}\label{p-sc}
&&P_{\eta\mu} [u](x,t)=\sup_{s \ge t}\{u(x,t)-v(x)-\mu (u(x,s)-v(x))-\mu\eta(s-t)\},
\end{eqnarray}
and
\begin{eqnarray}\label{m-sc}
&&\\\nonumber
&& M_{\eta\mu}[u](t)=\sup_{x \in \T^N,s \ge t}\{u(x,t)-v(x)-\mu (u(x,s)-v(x))-\mu\eta(s-t)\}
= \sup_{x \in \T^N} P_{\eta\mu} [u](x,t).
\end{eqnarray}

\begin{lem}\label{lem:ineq-var}
The function $P_{\eta\mu} [u](x,t)$ is a subsolution of the 
Hamilton-Jacobi inequality
\begin{eqnarray}\label{ineq-var}
&& {\rm min}
\left\{ U(x,t)\, , \,
 \frac{\partial U}{\partial t}
+\mathop{\rm inf}_{\theta\in\Theta}\big\{
-{\rm trace}(A_\theta(x)D^2U)\big\}
-C|DU|\right\}
\le 0
\quad \text{in } \T^N\times (0,+\infty),
\end{eqnarray}
where $C$ is a constant independent of $x,t$ (given in~\eqref{valC}). 
\end{lem}

\begin{proof}[Proof of Lemma~\ref{lem:ineq-var}]
For simplicity, we set $U(x,t):= P_{\eta\mu} [u](x,t).$

Let any $\phi_0 \in C^2(\T^N\times(0,\infty))$ such that $(x_0,t_0),$ $t_0>0,$ 
is a strict maximum point of $U-\phi_0$ in $\T^N\times[t_0-\delta,t_0+\delta]$ for some 
small $\delta >0$. If $U(x_0,t_0)\leq 0$, then \eqref{ineq-var} is automatically satisfied. 
We therefore assume that $U(x_0,t_0)>  0$ 
to continue. 

For $x,y,z \in \T^N$ and $0\leq t \leq s,$ we consider 
\begin{eqnarray}\label{defPhi123}
\Phi(x,y,z,t,s)=u(x,t)-v(z)-\mu (u(y,s)-v(z))
- \phi(x,y,z,t,s),
\end{eqnarray}
with
\begin{eqnarray}\label{defPhi124}
&&\\\nonumber
\phi(x,y,z,t,s)=\mu \eta(s-t)+\alpha^2(|x-y|^2+|x-z|^2+|y-z|^2)+|s-s_0|^2+\phi_0(x,t),
\end{eqnarray}
where $s_0$ is the  point where the maximum is achieved in \eqref{p-sc}. 
The function $\Phi$ achieves its maximum over $(\T^N)^3 \times  
\{(t,s) : s\geq t,\; t \in [t_0-\delta,t_0+\delta]$\} at  
$(\bar{x},\bar{y},\bar{z},\bar{t},\bar{s})$ because $u,v$ are bounded continuous. 
We obtain some classical estimates when $\alpha \to \infty,$
\begin{eqnarray}\label{Es3}
&& \left\{
\begin{array}{ll}
\Phi(\bar{x},\bar{y},\bar{z},\bar{t},\bar{s}) \to U(x_0,t_0)-\phi_0(x_0,t_0),\\[2mm]
\alpha(\bar{x}-\bar{y}), \, \alpha(\bar{x}-\bar{z})\, \alpha(\bar{y}-\bar{z})\,\to 0,\\[2mm]
(\bar{x},\bar{t},\bar{s}) \to (x_0,t_0,s_0)
\text{ since $(x_0,t_0,s_0)$ is a strict maximum point}\\
\hspace*{8cm}\text{of $U(x,t)-\phi_0(x,t)-|s-s_0|^2,$}\\[2mm]
\bar{s}>\bar{t} \text{ since } U(x_0,t_0)>  0.
\end{array}
\right.
\end{eqnarray}

In the sequel,
all the derivatives of $\phi$ are calculated at
$(\bar{x},\bar{y},\bar{z},\bar{t},\bar{s})$
so we skip this dependence for simplicity.

The theory of second order viscosity \cite{cil92} yields, for every
$\alpha>1,$ the existence of symmetric matrices $X,Y,Z$ such that
\begin{eqnarray}
&& (\phi_t,D_x\phi,X) \in \overline{J}^{2,+}u(\overline{x},\bar{t}),
\quad
(\frac{-\phi_s}{\mu},\frac{-D_y\phi}{\mu},\frac{-Y}{\mu}) 
\in \overline{J}^{2,-}u(\overline{y},\bar{s}),\label{subdiff}\\
&&
(\frac{D_z\phi}{\mu -1}, \frac{Z}{\mu -1})\in \overline{J}^{2,+}v(\overline{x}),\nonumber\\
&& 
-(\alpha^2 +|A|)I\leq
\left(
\begin{array}{ccc}
X & 0 & 0\\
0 & Y & 0\\
0 & 0 & Z
\end{array}
\right)
\leq A+ \frac{1}{\alpha^2}A^2,
\quad A=D^2\phi(\bar{x},\bar{y},\bar{z},\bar{t},\bar{s}).
\label{ineq-mat123}
\end{eqnarray}

We have
\begin{eqnarray}\label{mat_second}
A=2 \alpha^2\mathcal{A}+\mathcal{B},
\quad
\mathcal{A}:=
\left(
\begin{array}{ccc}
2I & -I & -I\\
-I & 2I & -I\\
-I & -I & 2I
\end{array}
\right),
\
\mathcal{B}:=
\left(
\begin{array}{ccc}
B & 0 & 0\\
0 & 0 & 0\\
0 & 0 & 0
\end{array}
\right),
\end{eqnarray}
where $B=D^2\phi_0(\bar{x},\bar{t}).$
It follows from~\eqref{ineq-mat123} and~\eqref{mat_second} that
\begin{eqnarray}\label{ineq-precisee}
&& -C(\alpha^2 \!+\!|B|)I
\leq
{\scriptsize
\left(
\begin{array}{ccc}
X & 0 & 0\\
0 & Y & 0\\
0 & 0 & Z
\end{array}
\right)}
\leq 
C\alpha^2\mathcal{A}
\!+\!
\mathcal{B}
+C(\frac{|B|^2}{\alpha^2}I\!+\!\mathcal{A}\mathcal{B}+\mathcal{B}\mathcal{A}).
\end{eqnarray}

Since $\tilde{u}$ is solution of~\eqref{DHJE} and $v$ is solution of~\eqref{sta},  
the following viscosity inequalities hold,
\begin{eqnarray}\label{visco-ineq1}
&& \left\{
\begin{array}{ll}
-\mu\eta+\displaystyle\frac{\partial \phi_0}{\partial t}(\bar{x},\bar{t})\\
\displaystyle
\hspace*{1cm}+\mathop{\rm sup}_{\theta\in\Theta}
\left\{-{\rm trace}(A_\theta(\bar{x}) X)+H_\theta\big(\bar{x},p+q
+D\phi_0(\bar{x},\bar{t})\big)\right\}
\le c,\\[2mm]
\displaystyle
-\eta+ 2(\bar{s}-s_0)+
\mathop{\rm sup}_{\theta\in\Theta}
\{-{\rm trace}(A_\theta(\bar{y})\frac{-Y}{\mu})+H_\theta(\bar{y},\frac{p}{\mu})\} \ge c,\\[2mm]
\displaystyle
\mathop{\rm sup}_{\theta\in\Theta}
\{-{\rm trace}(A_\theta(\bar{z})\frac{Z}{\mu-1})+H_\theta(\bar{z},\frac{-q}{\mu-1})\} \le c,
\end{array}
\right.
\end{eqnarray}
where  
\begin{eqnarray}\label{defpq}
p=2\alpha^2(\bar{x}-\bar{y})+2\alpha^2(\bar{z}-\bar{y}) 
\quad \text{and} \quad q=2\alpha^2(\bar{x}-\bar{z})+2\alpha^2(\bar{y}-\bar{z}).
\end{eqnarray}

In the sequel,
$o(1)\to 0$ as $\alpha\to +\infty$ uniformly with respect to $\theta.$

Using~\eqref{Hunif-cont},~\eqref{Es3} and the boundedness
of $|p|, |q|$
since $u,v$ are Lipschitz continuous with respect to $x$
(see \eqref{sol-lip-hold}), it follows
\begin{eqnarray}\label{meuf1}
&& \left\{
\begin{array}{ll}
\displaystyle\frac{\partial \phi_0}{\partial t}(x_0,t_0)
-\mu\eta\\
\displaystyle
\hspace*{1cm}+\mathop{\rm sup}_{\theta\in\Theta}
\{-{\rm trace}(A_\theta(\bar{x}) X)+H_\theta(x_0,p+q)\}
-C|D\phi_0(x_0,t_0)|\le c+o(1),\\[2mm]
\displaystyle-\mu\eta+\mathop{\rm sup}_{\theta\in\Theta}
\{{\rm trace}(A_\theta(\bar{y})Y)+\mu H_\theta(x_0,\frac{p}{\mu})\} \ge \mu c+o(1),\\[2mm]
\displaystyle\mathop{\rm sup}_{\theta\in\Theta}
\{-{\rm trace}(A_\theta(\bar{z})Z)+(\mu-1)H_\theta(x_0,\frac{-q}{\mu-1})\} \le (\mu-1)c+o(1).
\end{array}
\right.
\end{eqnarray}
Notice that the above constant $C$ may be chosen as
\begin{eqnarray}\label{valC}
&& C=  \mathop{\rm sup}_{\theta\in \Theta}
|(H_\theta)_p|_{L^\infty(\T^N\times R)},
\quad \text{ with } R= \mathop{\rm sup}_{t\geq 0}
|Du(\cdot,t)|_{L^\infty(\T^N)}+|Dv|_{L^\infty(\T^N)}.
\end{eqnarray}
Summing the inequalities leads to
\begin{eqnarray}\label{meuf3}
&&\frac{\partial \phi_0}{\partial t}(x_0,t_0)+\mathop{\rm inf}_{\theta\in\Theta}
\big\{-{\rm trace}(A_\theta(\bar{x}) X
+A_\theta(\bar{y})Y+A_\theta(\bar{z}) Z)\\
&&+H_\theta(x_0,p+q)
+(\mu-1)H_\theta(x_0,\frac{-q}{\mu-1})- \mu H_\theta(x_0,\frac{p}{\mu})\big\}
-C|D\phi_0(x_0,t_0)|
\le o(1).\nonumber
\end{eqnarray}

From~\eqref{mat_second} and \eqref{ineq-precisee}, 
using classical computations~\cite[p.74]{il90}, we obtain
\begin{eqnarray}\label{calcul-class123}
&&\\\nonumber
&& {\rm trace}(A_\theta(\bar{x}) X
+A_\theta(\bar{y})Y+A_\theta(\bar{z}) Z)\\\nonumber
&\le&
C\alpha^2 \, {\rm trace}\big(
(\sigma_\theta(\bar{x})-\sigma_\theta(\bar{y}))^T(\sigma_\theta(\bar{x})-\sigma_\theta(\bar{y}))
+(\sigma_\theta(\bar{x})-\sigma_\theta(\bar{z}))^T(\sigma_\theta(\bar{x})-\sigma_\theta(\bar{z}))\\
\nonumber
&& 
\hspace*{9cm}
+(\sigma_\theta(\bar{z})-\sigma_\theta(\bar{y}))^T(\sigma_\theta(\bar{z})-\sigma_\theta(\bar{y}))
\big)\\\nonumber
&& + {\rm trace}\left(A_\theta(\bar{x})B\right)
+C\frac{|B|^2}{\alpha^2}
+
{\rm trace}\left((2A_\theta(\bar{x})
\!-\!A_\theta(\bar{y})
\!-\!A_\theta(\bar{z}))B\right).
\end{eqnarray}
Since $B= D^2\phi_0(x_0,t_0),$ $|B|\leq C,$
$\sigma_\theta$ is Lispchitz continuous by~\eqref{H1_dege} and
$\bar{x},\bar{y},\bar{z}\to x_0,$
we obtain
\begin{eqnarray}\label{neuf2}
-{\rm trace}(A_\theta(\bar{x}) X
+A_\theta(\bar{y})Y+A_\theta(\bar{z}) Z) \ge  
-{\rm tr}(A_\theta(x_0)D^2\phi_0(x_0,t_0))
+o(1).
\end{eqnarray}

Since $H_\theta$ is convex and $p/\mu = (p+q)/\mu -q/\mu,$ we have
\begin{eqnarray}\label{Hcvx123}
H_\theta(x_0,p+q)
+(\mu-1)H_\theta(x_0,\frac{-q}{\mu-1})- \mu H_\theta(x_0,\frac{p}{\mu})\geq 0.
\end{eqnarray}

Using these previous estimates for \eqref{meuf3} and
letting $\alpha\to +\infty$, we obtain
\begin{eqnarray*}
&&
\frac{\partial \phi_0}{\partial t}(x_0,t_0)
+\mathop{\rm inf}_{\theta\in\Theta}\big\{
-{\rm tr}(A_\theta(x_0)D^2\phi_0(x_0,t_0))\big\}
-C|D\phi_0(x_0,t_0)|\le 0,
\end{eqnarray*}
which is exactly what we need. 
\end{proof}

We set $M_{\eta \mu}^+[u](t)={\rm max}\{0, M_{\eta \mu}[u](t)\}.$

\begin{lem}\label{Mdecroissant}
The function $M_{\eta \mu}^+[u](t)$ is nonincreasing, so it converges
to some constant ${m}_{\eta\mu}\geq 0$ as $t\to +\infty.$
\end{lem}

\begin{proof}[Proof of Lemma \ref{Mdecroissant}]
At first, it is easy to check that $M_{\eta \mu}[u]$ is continuous
and, from Lemma~\ref{lem:ineq-var}, by classical computations, $M_{\eta \mu}[u]$
is a viscosity subsolution of 
\begin{eqnarray}\label{ineq-var375}
\min \{ M_{\eta\mu} [u](t), M_{\eta\mu} [u]'(t)\}\leq 0 
\quad \text{on $(0,+\infty).$}
\end{eqnarray}
Let $J=\{t\in [0,+\infty) : M_{\eta \mu}[u](t)>0 \}.$
If $J=\emptyset,$ then $M_{\eta \mu}^+[u](t)=0$ for all $t$
and the conclusion follows. If $J\not=\emptyset,$ then,
by continuity, there exists $t_0<t_1$ such that
$[t_0,t_1]\subset J.$ By~\eqref{ineq-var375}
$M_{\eta\mu} [u]'(t)\leq 0$ in the viscosity
sense on $(t_0, t_1).$ Therefore, $t\mapsto M_{\eta\mu} [u](t)$
is nonincreasing on $[t_0, t_1].$ Necessarily, ${\rm inf}\, J=0$
and $t\mapsto M_{\eta\mu} [u](t)$
is nonincreasing on $[0,{\rm sup}\, J).$
If ${\rm sup}\, J=+\infty,$ then $t\mapsto M_{\eta\mu} [u](t)>0$
is nonincreasing on $[0,+\infty)$ and the conclusion follows.
If ${\rm sup}\, J<+\infty,$ then $M_{\eta\mu}^+ [u](t)=0$
on $[{\rm sup}\, J, +\infty)$ and therefore the limit is 0.
\end{proof}

The strategy of the proof of Theorem~\ref{strict_co_dege} is to obtain
${m}_{\eta 1}= 0.$ An immediate consequence is that $t\mapsto u(x,t)$ 
is nondecreasing for every $x.$ 
The conclusion follows easily, see the end of this section.

So, from now on, 
we argue by contradiction assuming that
\begin{eqnarray}\label{contr-ceta}
{m}_{\eta 1} >0.
\end{eqnarray}

The following result makes the link between
${m}_{\eta\mu}$ and ${m}_{\eta 1}.$

\begin{lem}\label{IM4_dege}
For all $\epsilon >0,$
there exists $\mu_{\epsilon} > 1$ such that, for $1\leq \mu \leq \mu_{\epsilon},$ we have
\begin{eqnarray}\label{ineg528}
P_{\eta \mu}[u](t)\geq P_{\eta 1}[u](t)- \epsilon,
\qquad t\geq 0.
\end{eqnarray}
In particular, there exists $\mu_{\eta} >1$ such that 
for $1<\mu <\mu_{\eta},$ we have
\begin{eqnarray}\label{ineg529}
M_{\eta \mu}[u](t)\geq \frac{{m}_{\eta 1}}{2}>0,
\qquad t\geq 0.
\end{eqnarray}
\end{lem}

\begin{proof}[Proof of Lemma \ref{IM4_dege}]
Let $x\in\T^N,$ $t\geq 0.$ There exists $s_1\geq t$ such that
\begin{eqnarray*}
P_{\eta 1}[u](x,t) = u(x,t)-\mu u(x,s_1)-\mu \eta (s_1-t)
\geq u(x,t)-\mu u(x,s_1)\geq -C
\end{eqnarray*}
since $u$ is bounded. We deduce $\eta (s_1-t)\leq C.$
Therefore
\begin{eqnarray*}
&&P_{\eta \mu}[u](x,t)-P_{\eta 1}[u](x,t) \geq 
(\mu -1)(v(x)-u(x,s_1)-\eta (s_1-t))\geq -C(\mu -1).
\end{eqnarray*}
To prove~\eqref{ineg529}, it is enough to notice that, since
${m}_{\eta 1} >0$ by~\eqref{contr-ceta}, then $M_{\eta 1}[u](t)$
is positive nonincreasing and bigger to ${m}_{\eta 1}.$ It is
then sufficient to choose $\epsilon= {m}_{\eta 1}/2.$ 
\end{proof}

From now on, we choose $1<\mu <\mu_\eta,$ where $\mu_\eta$ is given
by Lemma~\ref{IM4_dege}, in order that $M_{\eta \mu}[u](t)>0.$

\begin{lem}\label{Pvindept}
There exists $t_n\to +\infty$
such that $u(\cdot, \cdot +t_n)$ converges 
in $W^{1,\infty}(\T^N\times [0,+\infty))$
to a solution $\tilde{u}$ of~\eqref{DHJE}.
The function $P_{\eta\mu} [\tilde{u}](x,t)$ is a still
a subsolution of~\eqref{ineq-var}
and $M_{\eta\mu} [\tilde{u}](t)= {m}_{\eta \mu}>0$ is independent of $t.$
\end{lem}

\begin{proof}[Proof of Lemma \ref{Pvindept}]
By~\eqref{sol-lip-hold}, $\{u(\cdot,t), t\geq 0\}$ is relatively
compact in $W^{1,\infty}(\T^N).$ Let any sequence $t_n\to +\infty$
such that $u(\cdot, t_n)$ converges. By the comparison principle
for~\eqref{DHJE}, we have, for any $n,p\geq 0,$
\begin{eqnarray*}
|u(x,t+t_n)-u(x,t+t_p)|\leq |u(\cdot ,t_n)-u(\cdot ,t_p)|_\infty,
\qquad x\in\T^N, t\geq 0.
\end{eqnarray*}
Therefore $(u(\cdot, \cdot +t_n))_{n}$ is a Cauchy sequence
in $W^{1,\infty}(\T^N\times [0,+\infty)).$ So it converges to some function 
$\tilde{u}\in W^{1,\infty}(\T^N\times [0,+\infty)),$ which is still a solution
of~\eqref{DHJE} by classical stability results.

We observe that
\begin{eqnarray*}
&& M_{\eta\mu }[u](t+t_n)=\sup_{x \in \T^N,s \ge t}\{u(x,t+t_n)-v(x)-\mu 
(u(x,s+t_n)-v(x))-\mu\eta(s-t)\}.
\end{eqnarray*}
Since $M_{\eta \mu}[u](t+t_n)\to {m}_{\eta \mu}$ as $n\to +\infty,$
$M_{\eta \mu}[\tilde{u}](t)= {m}_{\eta \mu}$ is independent of $t.$ 

Finally, since $P_{\eta \mu}[u](x,t+t_n)$ converges uniformly to
$P_{\eta \mu}[\tilde{u}](x,t)$ as $n\to +\infty,$ we obtain that
$P_{\eta \mu}[\tilde{u}]$ is still a subsolution of~\eqref{ineq-var}
\end{proof}

\begin{lem}\label{rameneF}
For any $\tau >0,$
\begin{eqnarray}\label{defxtau}
\mathop{\rm max}_{x \in \T^N,t\geq 0} P_{\eta\mu}[\tilde{u}](x,t)
= M_{\eta\mu} [\tilde{u}](\tau)= {m}_{\eta \mu}= P_{\eta\mu}[\tilde{u}](x_\tau,\tau)
\quad \text{for $x_\tau \in \Sigma$ if $\Sigma\not=\emptyset.$}
\end{eqnarray}
If $\Sigma=\emptyset,$ then ${m}_{\eta 1}=0.$
\end{lem}

The point in this result is that the maximum
of $P_{\eta\mu}[\tilde{u}](\tau)$ is achieved at some point 
$x_\tau\in\Sigma.$

\begin{proof}[Proof of Lemma \ref{rameneF}]
Let $\tau >0$ and suppose that $x_\tau$ defined by~\eqref{defxtau}
lies in $\T^N-\Sigma.$ We write $U(x,t)= P_{\eta\mu}[\tilde{u}](x,t)$
for simplicity.
Since  $M_{\eta\mu} [\tilde{u}](t)$ is independent
of $t,$ we have
\begin{eqnarray*}
 U(x_\tau,\tau)
=\max_{x \in \T^N, \, t \ge 0}U(x,t)={m}_{\eta\mu}.
\end{eqnarray*}
We aim at applying the strong maximum principle
of Da Lio~\cite{dalio04} for viscosity solutions. Let $\delta >0$
and $\Sigma_\delta=\{{\rm dist}(\cdot, \Sigma)\geq \delta\}.$
We consider the connected component  $\mathcal{C}_\delta$ of  $x_\tau$
in $\Sigma_\delta^C\cap \{U(\cdot ,\tau)>0\}.$
From Lemmas~\ref{ineq-var} and~\ref{Pvindept}, $U$ is a subsolution
of 
\begin{eqnarray*}
 \frac{\partial U}{\partial t}
+G(x,DU,D^2U)\le 0
\qquad \text{in ${\rm int}(\mathcal{C}_\delta)\times (\tau-\tau/2,\tau+\tau/2),$}
\end{eqnarray*}
where
\begin{eqnarray*}
G(x,p,X)=\mathop{\rm inf}_{\theta\in\Theta}\big\{
-{\rm tr}(A_\theta(x)X)\big\}
-C|p|.
\end{eqnarray*}
From~\eqref{inver_dege},
\begin{eqnarray*}
G(x,p,X+Y)-G(x,p,Y)\leq -\nu_\delta \, {\rm trace}(Y),
\quad X,Y\in\mathcal{S}_N, \, Y\geq 0, \, x\in  \mathcal{C}_\delta, \, p\in \R^N
\end{eqnarray*}
and $G(x,\lambda p,\lambda X)=\lambda G(x,p,X)$ for $\lambda >0.$
From~\cite[Th. 2.1]{dalio04}, we infer that 
$U$ is constant and equal to ${m}_{\eta\mu}$
in $\partial(\mathcal{C}_\delta\times \{t=\tau\}).$
Moreover, since ${m}_{\eta\mu}>0,$ 
necessarily $\partial\mathcal{C}_\delta\subset \Sigma_\delta.$
It follows that there exists $x_{\tau\delta}$
such that $U(x_{\tau\delta},\tau)={m}_{\eta\mu}$
and ${\rm dist}(x_{\tau\delta},\Sigma)=\delta.$
Letting $\delta\to 0$ and extracting subsequences if
necessary, we find $y_\tau\in  \partial \Sigma$
such that $U(y_\tau,\tau)=P_{\eta\mu}[\tilde{u}](y_\tau,\tau)={m}_{\eta\mu}.$

In the case $\Sigma =\emptyset,$ we obtain that
$P_{\eta\mu}[\tilde{u}](x,t)={m}_{\eta\mu}$ in $\T^N\times [0,+\infty).$
Letting $\mu\to 1,$ we get
$P_{\eta 1}[\tilde{u}](x,t)={m}_{\eta 1}>0$ in $\T^N\times [0,+\infty).$
Let $s(t)$ be the point where the maximum is achieved 
in $P_{\eta 1}[\tilde{u}](x,t).$ We have
\begin{eqnarray*}
P_{\eta 1}[\tilde{u}](x,t)+P_{\eta 1}[\tilde{u}](x,s(t))
=2{m}_{\eta 1} = u(x,t)-u(x,s(s(t)))-\eta (s(s(t))-t)\leq {m}_{\eta 1}
\end{eqnarray*}
which leads to a contradiction with~\eqref{contr-ceta}
and implies ${m}_{\eta 1}=0.$
\end{proof}

We now obtain the desired contradiction with~\eqref{contr-ceta}.
The following result is is one the key step in the proof
of Theorem~\ref{strict_co_dege}.

\begin{lem}\label{IMpo1}
If, for some $\tau >0,$
\begin{eqnarray*}
\mathop{\rm max}_{x \in \T^N,t\geq 0} P_{\eta\mu}[\tilde{u}](x,t)
= M_{\eta\mu} [\tilde{u}](\tau)= {m}_{\eta \mu}= P_{\eta\mu}[\tilde{u}](x_\tau,\tau)
\quad \text{for $x_\tau \in \Sigma,$}
\end{eqnarray*}
then ${m}_{\eta \mu}=0.$
\end{lem}

\begin{proof}[Proof of Lemma \ref{IMpo1}]
We fix $\tau >0$ and we assume  that
\begin{eqnarray*}
&& {m}_{\eta \mu}=M_{\eta \mu} [\tilde{u}](\tau)
=\tilde{u}(x_\tau,\tau)-v(x_\tau)-\mu (\tilde{u}(x_\tau,s_\tau)-v(x_\tau))
-\mu \eta(s_\tau-\tau),
\quad\text{with $x_\tau\in \Sigma,$}
\end{eqnarray*}
and we recall that, by contradiction, we assume ${m}_{\eta \mu}>0.$
Notice that
$\tau$ is a strict maximum point of $t \mapsto M_{\eta \mu}[\tilde{u}](t)-|t-\tau|^2$
in $(0,+\infty)$
since $M_{\eta \mu}[\tilde{u}](t)$ is constant.

We define $\Phi, \phi$ as in~\eqref{defPhi123}-\eqref{defPhi124}
by replacing $s_0$ with $s_\tau$ in $\phi$ and choosing
$\phi_0(x,t)= \langle x-x_\tau\rangle+|t-\tau|^2,$
where $\langle x\rangle= \sqrt{\epsilon^2 +|x|^2}$ for some fixed $\epsilon >0.$

Exactly as in the proof of Lemma~\ref{ineq-var},
the function $\Phi$ achieves its maximum over $(\T^N)^3 \times  
\{(t,s) : s\geq t,\; t \in [t_0-\delta,t_0+\delta]$\} at  
$(\bar{x},\bar{y},\bar{z},\bar{t},\bar{s})$ and~\eqref{Es3} are replaced
with
\begin{eqnarray}\label{Es33}
&& \left\{
\begin{array}{ll}
\Phi(\bar{x},\bar{y},\bar{z},\bar{t},\bar{s}) \to  {m}_{\eta \mu}-\epsilon,\\[2mm]
\alpha(\bar{x}-\bar{y}), \, \alpha(\bar{x}-\bar{z})\, , \,
\alpha(\bar{y}-\bar{z})\,\to 0,\\[2mm]
(\bar{x},\bar{y},\bar{z},\bar{t},\bar{s}) \to (x_\tau,x_\tau,x_\tau,\tau,s_\tau)\\[2mm]
\bar{s}>\bar{t} \text{ since } M_{\eta \mu}[\tilde{u}](\tau)={m}_{\eta \mu}>  0.
\end{array}
\right.
\end{eqnarray}
Formulas~\eqref{subdiff}--\eqref{ineq-precisee} still hold
with $B=D_{xx}^2 \langle \cdot -x_\tau\rangle (\bar{x})
= \langle \bar{x} -x_\tau\rangle^{-1}
(I-\frac{\bar{x}-x_\tau}{ \langle \bar{x} -x_\tau\rangle}
\otimes \frac{\bar{x}-x_\tau}{ \langle \bar{x} -x_\tau\rangle}).$
Noticing that $|B|\leq \epsilon^{-1},$ we may refine~\eqref{ineq-precisee} 
\begin{eqnarray}\label{nlle-ineq}
&& -C(\alpha^2+\frac{1}{\epsilon}) I
\leq
\left(
\begin{array}{ccc}
X & 0 & 0\\
0 & Y & 0\\
0 & 0 & Z
\end{array}
\right)
\leq 
C\alpha^2
\left(
\begin{array}{ccc}
2I & -I & -I\\
-I & 2I & -I\\
-I & -I & 2I
\end{array}
\right)
+C(\frac{1}{\alpha^2\epsilon^2}+\frac{1}{\epsilon}) I.
\end{eqnarray}

In the sequel, $o(1)$ denotes a function which tends to 0 as $\alpha\to +\infty$
for fixed $\epsilon>0,$ uniformly with respect to $\theta.$

The viscosity inequalities~\eqref{visco-ineq1} and \eqref{meuf1}
hold with
$\frac{\partial \phi_0}{\partial t}(\bar{x},\bar{t})=2(\bar{t}-\tau),$
$D\phi_0(\bar{x},\bar{t})=\frac{\bar{x}-x_\tau}{ \langle \bar{x} -x_\tau\rangle}=o(1),$
\begin{eqnarray}\label{neuf1234}
\displaystyle
&& \left\{
\begin{array}{ll}
\displaystyle-\mu\eta+\mathop{\rm sup}_{\theta\in\Theta}
\{
-{\rm trace}(A_\theta(\bar{x}) X)+H_\theta(x_\tau,p+q)\}\le o(1),\\[3mm]
\displaystyle
-\mu\eta+\mathop{\rm sup}_{\theta\in\Theta}
\{
{\rm trace}(A_\theta(\bar{y})Y)+\mu H_\theta(x_\tau,\frac{p}{\mu})\} \ge o(1),\\[3mm]
\displaystyle
\mathop{\rm sup}_{\theta\in\Theta}
\{-{\rm trace}(A_\theta(\bar{z})Z)+(\mu-1)H_\theta(x_\tau,\frac{-q}{\mu-1})\} \le o(1),
\end{array}
\right.
\end{eqnarray}
with $p,q$ defined in~\eqref{defpq},
and~\eqref{meuf3} reads now
\begin{eqnarray}\label{neuf3}
&&\mathop{\rm inf}_{\theta\in\Theta}
\{-{\rm trace}(A_\theta(\bar{x}) X
+A_\theta(\bar{y})Y+A_\theta(\bar{z}) Z)
+ \mathcal{H}_\theta \}\le o(1).\nonumber
\end{eqnarray}
where we set
\begin{eqnarray*}
&& \mathcal{H}_\theta :=
H_\theta(x_\tau,p+q)+(\mu-1)H_\theta(x_\tau,\frac{-q}{\mu-1})
-\mu H_\theta(x_\tau,\frac{p}{\mu}).
\end{eqnarray*}
From~\eqref{calcul-class123}, we get
\begin{eqnarray}\label{neuf23}
\mathop{\rm inf}_{\theta\in\Theta}\{{-\rm trace}(A_\theta(\bar{x}) X
+A_\theta(\bar{y})Y+A_\theta(\bar{z}) Z)\} \ge o(1).
\end{eqnarray}
It then follows
\begin{eqnarray}\label{neuf8}
\mathop{\rm inf}_{\theta\in\Theta}\mathcal{H}_\theta \le o(1).
\end{eqnarray}
From the convexity of $H_\theta,$ we know that $\mathcal{H}_\theta\geq 0$
(see~\eqref{Hcvx123}) but we need a strict inequality to
reach a contradiction.

Up to extract subsequences, we may assume that
$$
\lim_{\alpha \to \infty} p = \bar{p}
\quad {\rm and} \quad
\lim_{\alpha \to \infty} q = \bar{q},
$$
(recall that $p$ and $q$ are given by~\eqref{defpq} and are bounded
since $\tilde{u}, v$ are Lipschitz continuous).
We distinguish two cases depending on the above limit.

\noindent{\it First case.} We suppose that
$$
\frac{\bar{p}}{\mu}+\frac{\bar{q}}{\mu-1} \neq 0.
$$
Letting $\alpha\to +\infty$ in~\eqref{neuf8}
and recalling that $x_\tau\in\Sigma,$ we obtain
a contradiction thanks to the strict convexity of $H_\theta.$
More precisely, we apply~\eqref{cvx_neuf}
with $\lambda:=1/\mu$ and $P\not= Q$ given by
$P:=\bar{p}+\bar{q},$ $Q:=-\bar{q}/(\mu -1).$

\noindent{\it Second case.}
One necessarily has
\begin{eqnarray}\label{defp0}
\frac{\bar{p}}{\mu}=\frac{-\bar{q}}{\mu-1}=:p_\e=p_\e(\eta,\mu,\epsilon).
\end{eqnarray}
Notice that, in this case,
$\lim_{\alpha \to \infty}\mathcal{H}_\theta=0$ and therefore the strict 
convexity of the $H$ does not play any role.

From~\eqref{nlle-ineq}, we have
$|X|, |Y|, |Z|\le C(\alpha^2+(\alpha\e)^{-2}+\epsilon^{-1})$. Hence
\begin{eqnarray}\label{estimX}
|{\rm trace}(\sigma_\theta(\bar{x})\sigma_\theta(\bar{x})^T X) |
&\le& |\sigma_\theta(\bar{x})|^2 |X|
= |\sigma_\theta(\bar{x})-\sigma_\theta(x_\tau)|^2|X|\\\nonumber
&\le& C(\alpha^2+\frac{1}{\alpha^2\e^2}+\frac{1}{\epsilon})|\bar{x}-x_\tau|^2,
\end{eqnarray}
where we used  the fact that $\sigma(x_\tau)=0$
since $x_\tau\in \Sigma.$

We estimate the rate of convergence of the term $|\bar{x}-x_\tau|.$ Since $\Phi$ 
achieves its maximum  at  $(\bar{x},\bar{y},\bar{z},\bar{t},\bar{s})$, we have
\begin{eqnarray*}
\tilde{u}(\bar{x},\bar{t})-v(\bar{z})-\mu (\tilde{u}(\bar{y},\bar{s})-v(\bar{z}))
-\mu\eta(\bar{s}-\bar{t})-\langle \bar{x}-x_\tau\rangle
\geq \Phi (\bar{x},\bar{y},\bar{z},\bar{t},\bar{s})
\geq  M_{\eta,\mu}[\tilde{u}](\tau)-\epsilon.
\end{eqnarray*}
This implies
\begin{eqnarray*}
\sqrt{\epsilon^2+|\bar{x}-x_\tau|^2}&=&\langle \bar{x}-x_\tau\rangle\\
&\le& 
\tilde{u}(\bar{x},\bar{t})-v(\bar{z})-\mu (\tilde{u}(\bar{y},\bar{s})-v(\bar{z}))
-\mu\eta(\bar{s}-\bar{t})-M_{\eta,\mu}[\tilde{u}](\tau)+\epsilon\\
&=& [\tilde{u}(\bar{x},\bar{t})-\tilde{u}(\bar{z},\bar{t})]
+\mu[\tilde{u}(\bar{z},\bar{s})-\tilde{u}(\bar{y},\bar{s})]\\
&& +[\tilde{u}(\bar{z},\bar{t})-v(\bar{z})-\mu(\tilde{u}(\bar{z},\bar{s})-v(\bar{z}))
-\mu\eta(\bar{s}-\bar{t})]-M_{\eta,\mu}[\tilde{u}](\bar{t})+\epsilon\\
&\le& [\tilde{u}(\bar{x},\bar{t})-\tilde{u}(\bar{z},\bar{t})]
+\mu[\tilde{u}(\bar{z},\bar{s})-\tilde{u}(\bar{y},\bar{s})]+\epsilon\\
& \le& C(|\bar{x}-\bar{y}|+|\bar{x}-\bar{z}|)+\epsilon,
\end{eqnarray*}
where we used the fact that $M_{\eta,\mu}[\tilde{u}](t)={m}_{\eta\mu}$ for all 
$t>0$ and $\tilde{u}$ is Lipschitz continuous.
So,
\begin{eqnarray*}
|\bar{x}-x_\tau|^2 \le C(|\bar{x}-\bar{y}|^2+|\bar{x}-\bar{z}|^2)
+C\epsilon(|\bar{x}-\bar{y}|+|\bar{x}-\bar{z}|).
\end{eqnarray*}
It is worth noticing that $C$ depends only on $\tilde{u}$. Recalling that
$\alpha^2|\bar{x}-\bar{y}|,$ $\alpha^2|\bar{x}-\bar{z}|$ are bounded and
plugging the above estimates
in~\eqref{estimX}, we get
\begin{eqnarray*}
&&
{\rm trace}(A_\theta(\bar{x})X)= o(1)+O(\epsilon),
\end{eqnarray*}
where, for fixed $\epsilon >0,$ $o(1)\to 0$ as $\alpha\to +\infty$ and 
$O(\epsilon)\to 0$ as $\epsilon\to 0.$
Both error terms are uniform in $\theta.$
In the same way, we obtain
\begin{eqnarray*}
&&{\rm trace}(A_\theta(\bar{y})Y)\, , \,
{\rm trace}(A_\theta(\bar{z})Z)=o(1)+O(\epsilon).
\end{eqnarray*}
Sending $\alpha$ to $+\infty$ in~\eqref{neuf1234}, we have
\begin{eqnarray*}
&& \left\{
\begin{array}{ll}
\displaystyle
-\mu\eta+\mathop{\rm sup}_{\theta\in\Theta}\{H_\theta(x_\tau,p_\e)\}+O(\epsilon)\le 0,\\[2mm]
\displaystyle
-\mu\eta
+\mathop{\rm sup}_{\theta\in\Theta}\{\mu H_\theta(x_\tau,p_\e)\}+O(\epsilon) \ge 0,\\[2mm]
\displaystyle
\mathop{\rm sup}_{\theta\in\Theta}\{(\mu-1)H_\theta(x_\tau,p_\e)\}+O(\epsilon) \le 0
\end{array}
\right.
\end{eqnarray*}
(we recall that $p_\e$ is defined in~\eqref{defp0}).
Up to a subsequence if necessary, we can assume that $p_\e \to p_0$ when $\e \to 0$. 
So, we get
\begin{eqnarray*}
&& \left\{
\begin{array}{ll}
\displaystyle-\mu\eta+\mathop{\rm sup}_{\theta\in\Theta}\{H_\theta(x_\tau,p_0)\}\le 0,\\[2mm]
\displaystyle-\mu\eta
+\mathop{\rm sup}_{\theta\in\Theta}\{\mu H_\theta(x_\tau,p_0)\} \ge 0,\\[2mm]
\displaystyle\mathop{\rm sup}_{\theta\in\Theta}\{(\mu-1)H_\theta(x_\tau,p_0)\} \le 0.
\end{array}
\right.
\end{eqnarray*}
This implies $\mu\eta=0$, which is a contradiction.
It ends the proof.
\end{proof}

\begin{proof}[End of the proof of Theorem \ref{strict_co_dege}]
We obtained that ${m}_{\eta 1}=0.$
From ${m}_{\eta 1}=0,$ we infer
\begin{eqnarray*}
\tilde{u}(x,t)-\tilde{u}(x,s)-\eta(s-t)\le 0, \text{ for all $x \in \T^N$ and $s \ge t\geq 0.$}
\end{eqnarray*}
Letting $\eta$ tend to 0, we obtain
\begin{eqnarray*}
\tilde{u}(x,t)-\tilde{u}(x,s)\le 0.
\end{eqnarray*}
The uniform convergence of $(u(\cdot,t_n+\cdot))_{n}$ 
to $\tilde{u}\in W^{1,\infty}(\T^N\times [0,+\infty))$ (see  Lemma~\ref{Pvindept})
yields
\begin{eqnarray*}
-o_n(1)+\tilde{u}(x,t)\le u(x,t+t_n) \le o_n(1)+\tilde{u}(x,t)
\quad\text{in $\T^N \times (0,\infty)$}.
\end{eqnarray*}
Since $\tilde{u}$ is nondecreasing in $t$, 
there exists $u_\infty\in W^{1,\infty}(\T^N)$
such that $\tilde{u}(\cdot ,t) \to u_{\infty}(\cdot)$ uniformly 
as $t$ tends to infinity.
Taking Barles-Perthame half relaxed limits, we obtain
\begin{eqnarray*}
-o_n(1)+u_{\infty}(x)\le \liminf_{t\to +\infty}\phantom{ }_*\, 
u(x,t)\le \limsup_{t\to +\infty}\phantom{ }^*\,u(x,t) \le o_n(1)+u_{\infty}(x)
\quad x\in\T^N.
\end{eqnarray*}
Letting $n$ tend to infinity, we derive
\begin{eqnarray*}
\liminf_{t\to +\infty}\phantom{ }_*\, u(x,t)
=\limsup_{t\to +\infty}\phantom{ }^*\, u(x,t)=u_{\infty}(x), \quad x\in\T^N,
\end{eqnarray*}
which yields the uniform convergence of $u(\cdot ,t)$ to $u_{\infty}$ 
in $\T^N$ as $t$ tends to infinity. 

By the stability result, $u_{\infty}$ is a solution of \eqref{sta} 
with $c=0.$ It ends the 
proof of Theorem~\ref{strict_co_dege}.
\end{proof}


\section{Proof of Theorem~\ref{mainresult} and Proposition~\ref{valc-nr}}\label{mix_main}

The proof of Theorem~\ref{mainresult} follows the same ideas as the 
one of Theorem~\ref{strict_co_dege} with minor adaptations.
It is actually easier, since, from~\eqref{0sous-sol}, we choose $v=0$ 
in~\eqref{p-sc}-\eqref{m-sc} which allows to simplify several arguments. 
We only provide the proof of the main changes which consist, on the one side, 
in taking into account the set $K$ which appears in~\eqref{H10} and,
on the other side, in the proof of Lemma \ref{IMpo1}.

As in the proof of Theorem~\ref{strict_co_dege}, we start with a change
of function $u\to u+ct$ which allows to deal with bounded functions
$u, \tilde{u}$ and $c=0.$

\begin{lem}\label{ensK}
For every $x_0\in K,$
The function $t\mapsto u(x_0,t)$ is nonincreasing.
\end{lem}

\begin{proof}[Proof of Lemma \ref{ensK}]
Let $x_0\in K,$ $t_0\geq 0$ and we assume
by contradiction that there exists $s_0 >t_0$ such that $u(x_0,s_0)>u(x_0,t_0).$
Consider, for $\epsilon, \alpha >0,$
\begin{eqnarray}\label{1234-az}
\mathop{\rm sup}_{x\in\T^N, t\geq t_0}\{u(x,t)-u(x_0,t_0)-\frac{|x-x_0|^2}{\epsilon^2}
-\alpha (t-t_0)\}.
\end{eqnarray}
Since $u$ is bounded,
this supremum is positive and is achieved at some $(\bar{x},\bar{t})$ with $\bar{t}>t_0$
for $\epsilon, \alpha >0$ small enough.
By classical estimates, $\frac{|\bar{x}-x_0|^2}{\epsilon^2}\to 0$ as
$\epsilon\to 0.$ Since $u$ is a viscosity subsolution of~\eqref{DHJE},
we obtain
\begin{eqnarray}\label{ineqvis87}
\alpha+\mathop{\rm sup}_{\theta\in\Theta}
\{-{\rm trace}(A_\theta(\bar{x}) \frac{2I}{\epsilon^2})
+H_\theta(\bar{x},p)\}\le 0,
\end{eqnarray}
with $p=2\frac{\bar{x}-x_0}{\epsilon^2}.$ On the one side,
since $u(\cdot,t)$ is Lipschitz continuous,
$p$ is bounded and, up to extract a subsequence as $\epsilon\to 0,$ 
we may assume that $p\to\bar{p}.$
On the other side, since $\bar{x}\to x_0\in K\subset \Sigma$
and $\sigma_\theta$ satisfies~\eqref{H1_dege},
\begin{eqnarray*}
|{\rm trace}(A_\theta(\bar{x}) \frac{2I}{\epsilon^2}|
\leq \frac{|\sigma_\theta(\bar{x})|^2}{\epsilon^2}
\leq C\frac{|\bar{x}-x_0|^2}{\epsilon^2}.
\end{eqnarray*}
From~\eqref{ineqvis87}, sending $\epsilon\to 0,$ we obtain
\begin{eqnarray*}
\alpha+\mathop{\rm sup}_{\theta\in\Theta}
H_\theta(\bar{x},\bar{p})\le 0,
\end{eqnarray*}
which is a contradiction with~\eqref{H10}(ii)(a) (with $c=0$).

Therefore, for all $s_0\geq t_0,$ we have  $u(x_0,s_0)\leq u(x_0,t_0).$
\end{proof}

A consequence of  Lemma~\ref{ensK} is that $u(x,t)$ converges on 
$K$ and therefore
\begin{eqnarray*}
\tilde{u}(x,t)~~\text{is independent of $t$, for any $x \in K$,}
\end{eqnarray*}
where $\tilde{u}$ is defined in the statement of Lemma~\ref{Pvindept}.
Assuming, as in the proof of Theorem~\ref{strict_co_dege}, that
${m}_{\eta 1}>0$ (and therefore ${m}_{\eta\mu}>0$ for $\mu$ close to 1),
we obtain  from the very definition of $P_{\eta \mu} [\tilde{u}]$ that
\begin{eqnarray}\label{dist-a-K}
{\rm dist}(x_\tau,K) \not= 0 \text{ for $\mu$ close enough to 1,}
\end{eqnarray}
where $x_\tau\in\T^N$ is the point where the maximum is achieved
in $M_{\eta\mu}[\tilde{u}](\tau).$

\begin{proof}[Proof of Lemma \ref{IMpo1} under the assumptions 
of Theorem~\ref{mainresult}]
deptra
Let us note that Lemma \ref{rameneF} is still true under the assumptions 
of Theorem~\ref{mainresult}, so we can assume that
$
x_\tau\in \Sigma.
$

Since $v=0$ in~\eqref{p-sc}-\eqref{m-sc},
we may choose $Z=0$ in~\eqref{nlle-ineq}, and 
$q=0$ in~\eqref{defpq}. The viscosity inequalities~\eqref{neuf1234} reads
\begin{eqnarray}\label{1234-mr}
\displaystyle
&& \left\{
\begin{array}{ll}
\displaystyle-\mu\eta+\mathop{\rm sup}_{\theta\in\Theta}
\{
-{\rm trace}(A_\theta(\bar{x}) X)+H_\theta(x_\tau,p)\}\le o(1),\\[3mm]
\displaystyle
-\mu\eta+\mathop{\rm sup}_{\theta\in\Theta}
\{
{\rm trace}(A_\theta(\bar{y})Y)+\mu H_\theta(x_\tau,\frac{p}{\mu})\} \ge o(1),\\[3mm]
\displaystyle
\mathop{\rm sup}_{\theta\in\Theta}
H_\theta(x_\tau,0) \le o(1).
\end{array}
\right.
\end{eqnarray}
Notice that the third inequality is nothing than~\eqref{0sous-sol} (with $c=0$
after our change of function).
Subtracting the two first inequalities from~\eqref{neuf23} yield
\begin{eqnarray}\label{subtract-mr}
&& \mathop{\rm inf}_{\theta\in\Theta}\{
H_\theta(x_\tau,p)- \mu H_\theta(x_\tau,\frac{p}{\mu}) \}\le o(1).
\end{eqnarray}
As in the corresponding proof in Section~\ref{strict_dege}, 
we distinguish two cases depending on 
\begin{eqnarray*}
&& \mathop{\rm lim}_{\alpha\to +\infty} p=\bar{p}
\end{eqnarray*}
(up to subsequences if necessary).

\noindent{\it First Case.} If $\bar{p}\not= 0.$
Letting $\alpha\to +\infty$ in~\eqref{subtract-mr}
and recalling~\eqref{dist-a-K}, we obtain
a contradiction with~\eqref{H10}(ii)(b).

\noindent{\it Second Case.} If $\bar{p}= 0.$
Proceeding similarly as in the second case of the proof of Lemma~\eqref{IMpo1}, we obtain
\begin{eqnarray*}
|{\rm trace}(A_\theta(\bar{y}) Y) |
=o(1)+O(\e).
\end{eqnarray*}
Taking into account this estimate, by sending $\alpha\to \infty$ and then $\e \to 0$ in
the second inequality in~\eqref{1234-mr}
, we get
\begin{eqnarray*}
 -\mu\eta 
+\mathop{\rm sup}_{\theta\in\Theta}\{\mu H_\theta(x_\tau,0)\} \ge 0,
\end{eqnarray*}
which is a contradiction with the third inequality in~\eqref{1234-mr}.
\end{proof}

\begin{proof}[Proof of Proposition \ref{valc-nr}]
Consider the solution $v_\lambda^\epsilon$
of 
\begin{eqnarray*}
\lambda v_\lambda^\epsilon+\mathop{\rm sup}_{|e|\leq \epsilon,\,\theta\in\Theta}
\{-{\rm trace}(A_\theta(x+e)D^2 v_\lambda^\epsilon) + H_\theta(x+e,Dv_\lambda^\epsilon)\}=0, \quad
x\in \T^N.
\end{eqnarray*}
It follows from~\cite[Lemma 2.7]{bj02} that 
$v_{\lambda\epsilon}=\rho_\epsilon *  v_\lambda^\epsilon,$ where $\rho_\epsilon$ is a standard
mollifier, is a $C^\infty$ subsolution of~\eqref{DHJE-sta}. Moreover, 
from~\cite[Theorem A.1]{bj02}, we have $\lambda |v_\lambda-v_{\lambda\epsilon}|\leq C\epsilon.$
Therefore, we have
in the classical sense at any $x\in\T^N,$
\begin{eqnarray*}
&&\lambda v_{\lambda\epsilon}(x)+\mathop{\rm sup}_{\theta\in\Theta}
\{-{\rm trace}(A_\theta(x)D^2 v_{\lambda\epsilon}(x)) + H_\theta(x,Dv_{\lambda\epsilon}(x))\}\leq 0.
\end{eqnarray*}
We can write this inequality at any $\hat{x}\in \Sigma$ where 
${\rm trace}(A_\theta(\hat{x})D^2 v_{\lambda\epsilon}(\hat{x}))=0.$
It follows
\begin{eqnarray*}
&&-\lambda v_{\lambda}(\hat{x}) +C\epsilon 
\geq
-\lambda v_{\lambda\epsilon}(\hat{x})
\geq 
\mathop{\rm sup}_{\theta\in\Theta}H_\theta(\hat{x},Dv_{\lambda\epsilon}(\hat{x}))
\geq 
\mathop{\rm sup}_{\theta\in\Theta}H_\theta(\hat{x},0),
\end{eqnarray*}
 using~\eqref{hypHF}. Sending $\lambda\to 0$ and then $\epsilon\to 0,$
we obtain 
\begin{eqnarray*}
-\lambda v_{\lambda}(\hat{x})\to c\geq \mathop{\rm sup}_{\theta\in\Theta}H_\theta(\hat{x},0), 
\quad\text{for any $\hat{x}\in\Sigma.$}
\end{eqnarray*}
Hence $c\geq \mathop{\rm sup}_{x\in\Sigma,\theta\in\Theta}H_\theta(x,0)$.

We prove now the opposite inequality under either \eqref{Hpresque-cvx} or \eqref{supTS}. 

\noindent{\it Under Assumption~\eqref{supTS}.}
Let $v_\lambda$ be the solution of~\eqref{DHJE-sta} and $x_\lambda\in\T^N$
such that $v_\lambda (x_\lambda)={\rm min}_{\T^N}v_\lambda.$ We have
$\lambda v_\lambda (x_\lambda) + {\rm sup}_{\theta} H_\theta (x_\lambda, 0)\geq 0.$
Taking a subsequence $\lambda\to 0$ such that $\lambda v_\lambda\to -c,$
we get 
\begin{eqnarray*}
c\leq  \mathop{\rm sup}_{x\in \T^N, \theta\in\Theta}  H_\theta (x, 0)
=\mathop{\rm sup}_{x\in \Sigma, \theta\in\Theta}  H_\theta (x, 0)
\end{eqnarray*}
by~\eqref{hypHF}.

\noindent{\it Under Assumption~\eqref{Hpresque-cvx}.}
We set $\hat{H}_\theta (x,p)=H_\theta (x,p)-C$ where $C>0$ is 
big enough in order that $\hat{H}_\theta (x,0)\leq 0.$
It follows that, if $v_\lambda$ is a solution of~\eqref{DHJE-sta},
then $\hat{v}_\lambda=v_\lambda+C/\lambda$ is a solution of
\begin{eqnarray}\label{np11}
\lambda\hat{v}_\lambda +\mathop{\rm sup}_{\theta\in\Theta}
\{-{\rm trace}(A_\theta(x)D^2 \hat{v}_\lambda) 
+ \hat{H}_\theta(x,D\hat{v}_\lambda)\}=0
\end{eqnarray}
and $\hat{v}_\lambda\geq 0.$
For any $\gamma >1,$ we have
\begin{eqnarray*}
\frac{\lambda}{\gamma} \hat{v}_{\lambda/\gamma}+\mathop{\rm sup}_{\theta\in\Theta}
\{-{\rm trace}(A_\theta(x)D^2\hat{v}_{\lambda/\gamma}) 
+ \hat{H}_\theta(x,D\hat{v}_{\lambda/\gamma})\}=0,
\end{eqnarray*}
equivalently,
\begin{eqnarray*}
\lambda \hat{v}_{\lambda/\gamma}+\mathop{\rm sup}_{\theta\in\Theta}
\{-{\rm trace}(A_\theta(x)D^2(\gamma \hat{v}_{\lambda/\gamma})) 
+ \gamma \hat{H}_\theta(x,D\hat{v}_{\lambda/\gamma})\}=0.
\end{eqnarray*}
Noticing that $\hat{H}_\theta$ still satisfies~\eqref{Hpresque-cvx},
we have
\begin{eqnarray}\label{np22}
&& \lambda (1-\gamma) \min_{\T^N} \hat{v}_{\lambda/\gamma}
+\lambda (\gamma \hat{v}_{\lambda/\gamma}) \\\nonumber
&&+\mathop{\rm sup}_{\theta\in\Theta}
\{-{\rm trace}(A_\theta(x)D^2 (\gamma \hat{v}_{\lambda/\gamma})) 
+\hat{H}_\theta(x,D(\gamma \hat{v}_{\lambda/\gamma}))
-(1-\gamma) \hat{H}_\theta(x,0)\}\ge 0.
\end{eqnarray}
Subtracting~\eqref{np11} and~\eqref{np22}, we get, for 
$w_{\lambda\gamma}= \hat{v}_\lambda -\gamma  \hat{v}_{\lambda/\gamma},$
\begin{eqnarray*}
0 
&\geq &
\lambda (\gamma -1) \min_{\T^N} \hat{v}_{\lambda/\gamma}
+\lambda w_{\lambda\gamma}\\
&& +\mathop{\rm inf}_{\theta\in\Theta}\{-{\rm trace}(A_\theta D^2w_{\lambda\gamma}) 
+ \hat{H}_\theta(x,D\hat{v}_\lambda)-\hat{H}_\theta(x,D(\gamma \hat{v}_{\lambda/\gamma}))
+(1-\gamma) \hat{H}_\theta(x,0)\}\\
&\geq&
\lambda (\gamma -1) \min_{\T^N} \hat{v}_{\lambda/\gamma}
+ \lambda w_{\lambda\gamma}
+ \mathop{\rm inf}_{\theta\in\Theta}\{-{\rm trace}(A_\theta D^2w_{\lambda\gamma}) 
+(1-\gamma) \hat{H}_\theta(x,0)\}-C|Dw_{\lambda\gamma}|.
\end{eqnarray*}
Therefore
\begin{eqnarray}\label{np33}
&&\lambda w_{\lambda\gamma}
+ \mathop{\rm inf}_{\theta\in\Theta}\{-{\rm trace}(A_\theta D^2w_{\lambda\gamma})\}
-C|Dw_{\lambda\gamma}|
\leq (\gamma -1)
\left(  \mathop{\rm sup}_{\theta\in\Theta}\hat{H}_\theta(x,0)
-\lambda  \min_{\T^N} \hat{v}_{\lambda/\gamma}\right).
\end{eqnarray}
Recalling that $\hat{H}_\theta(x,0)\leq 0$ and $\hat{v}_{\lambda/\gamma}\geq 0,$
the right-hand side of~\eqref{np33} is nonnegative. By the strong maximum principle,
we obtain
\begin{eqnarray}\label{con_licit3}
\max_{x \in \T^N}w_{\lambda\gamma} =w_{\lambda\gamma}(x_0)
\quad \text{with $x_0 \in \Sigma$}.
\end{eqnarray}
Writing~\eqref{np33} at $x_0,$ we obtain
\begin{eqnarray*}
&&\lambda w_{\lambda\gamma}(x_0)
\leq (\gamma -1)
\left(  \mathop{\rm sup}_{x\in\Sigma, \theta\in\Theta}\hat{H}_\theta(x,0)
-\lambda  \min_{\T^N} \hat{v}_{\lambda/\gamma}\right).
\end{eqnarray*}
It follows
\begin{eqnarray*}
\lambda v_\lambda(x_0)-\gamma^2\frac{\lambda}{\gamma}v_{\lambda/\gamma}(x_0)
\leq (\gamma -1)
\left( \mathop{\rm sup}_{x\in\Sigma, \theta\in\Theta}H_\theta(x,0)
-\gamma \min_{\T^N} \frac{\lambda}{\gamma}v_{\lambda/\gamma}\right).
\end{eqnarray*}
Sending $\lambda\to 0,$ up to take subsequences, we obtain
\begin{eqnarray*}
c \leq \mathop{\rm sup}_{x \in \Sigma,\theta\in\Theta}
H_\theta(x,0).
\end{eqnarray*}
\end{proof}

\section{Proof of Theorem \ref{cas-degenere} and Propositions~\ref{ell-sur-sol}
and~\ref{cor-degenere}}
\label{sec:dege}

\begin{proof}[Proof of Theorem \ref{cas-degenere}]
The proof follows exactly the same line as those of
Theorems~\ref{strict_co_dege} and~\ref{mainresult}.
The only difference is the proof of Lemma~\ref{rameneF}
which is given below.
\end{proof}

\begin{proof}[Proof of Lemma \ref{rameneF} when~\eqref{sur-sol-stricte} holds.]
We write $U=P_{\eta\mu}[\tilde{u}]$ for simplicity.
Since $U$ is bounded, we can consider the half-relaxed limit
\begin{eqnarray*}
\overline{U}(x)= \mathop{\rm lim\,sup}_{y\to x, t\to +\infty}\, U(y,t).
\end{eqnarray*}
From Lemma~\ref{lem:ineq-var} and by the stability result, $\overline{U}$ is 
a viscosity subsolution of
\begin{eqnarray}\label{HJstatio111}
{\rm min}\{ \overline{U}\, , \, \mathop{\rm inf}_{\theta\in\Theta}\big\{
-{\rm tr}(A_\theta(x)D^2\overline{U})\big\}
-C|D\overline{U}|\} \le 0, \quad x\in \T^N.
\end{eqnarray}
Notice that one still has ${\rm max}_{\T^N}\overline{U}= {m}_{\eta\mu}>0.$

\noindent{\it Step 1. ${\rm argmax}\,\overline{U}\cap \Sigma\not=\emptyset$
thanks to~\eqref{sur-sol-stricte}.}
We argue by contradiction assuming that there exists $\delta >0$
such that ${\rm argmax}\,\overline{U}\subset
\Sigma_\delta^C,$ where $\Sigma_\delta=\{{\rm dist}(\cdot,\Sigma)\leq \delta\}.$
It follows that there exists $\rho_\delta >0$ such that
\begin{eqnarray} \label{abc246}
{m}_{\eta\mu}=\overline{U}(\hat{x})= \mathop{\rm max}_{\T^N}\overline{U}
=\mathop{\rm max}_{\Sigma_\delta^C}\overline{U}
\geq \mathop{\rm max}_{\Sigma_\delta}\overline{U}+\rho_\delta, \quad
\text{for some $\hat{x}\in \Sigma_\delta^C.$}
\end{eqnarray}
Let $\tilde{U}$ be a 1-periodic function of $\R^N$ such that
$\overline{U}(\pi(\tilde{x}))=\tilde{U}(\tilde{x})$ for all $\tilde{x}\in\R^N$
and $\tilde{\Sigma}_\delta=\{{\rm dist}(\cdot, \tilde{\Sigma})\leq \delta\}.$
From~\eqref{abc246} and by 1-periodicity, we infer
\begin{eqnarray*}
{m}_{\eta\mu}=\tilde{U}(\tilde{x})= \mathop{\rm max}_{\R^N}\tilde{U}
=\mathop{\rm sup}_{\tilde{\Sigma}_\delta^C}\tilde{U}
\geq \mathop{\rm sup}_{\tilde{\Sigma}_\delta}\tilde{U}+\rho_\delta, \quad
\text{for some $\tilde{x}\in \tilde{\Sigma}_\delta^C\cap [0,1]^N.$}
\end{eqnarray*}

For this $\delta>0,$ we consider the $C^2$ supersolution $\tilde{\psi}_\delta$
and $\Omega_\delta$ given 
by~\eqref{sur-sol-stricte}. Notice that, up to divide $\tilde{\psi}_\delta$
by a constant, we can assume that $|\tilde{\psi}_\delta|\leq 1$ in $\overline{\Omega}_\delta.$
We claim that, for $\varepsilon>0$ small enough,
\begin{eqnarray*}
\mathop{\rm sup}_{\R^N}\{ \tilde{U}-\varepsilon \tilde{\psi}_\delta\}=
 \tilde{U}(\tilde{x}_\delta)-\varepsilon \tilde{\psi}_\delta(\tilde{x}_\delta)
\quad\text{ with $\tilde{x}_\delta\in\tilde{\Sigma}_\delta^C\cap \Omega_\delta$ and
$\tilde{U}(\tilde{x}_\delta)>0.$}
\end{eqnarray*}
Indeed, using that $|\tilde{\psi}_\delta|\leq 1$ in $\overline{\Omega}_\delta$
and $\tilde{\psi}_\delta\geq 0$ on $\Omega_\delta^C,$ we have
\begin{eqnarray*}
\mathop{\rm sup}_{\R^N}\{ \tilde{U}-\varepsilon \tilde{\psi}_\delta\}
&\geq &
\tilde{U}(\tilde{x})-\varepsilon \tilde{\psi}_\delta(\tilde{x})\\
&\geq&
\mathop{\rm sup}_{\tilde{\Sigma}_\delta^C}\tilde{U}-\varepsilon \\
&\geq& \mathop{\rm sup}_{\tilde{\Sigma}_\delta}\tilde{U}+\rho_\delta -\varepsilon\\
&\geq& \mathop{\rm sup}_{\tilde{\Sigma}_\delta}\{ \tilde{U}-\varepsilon \tilde{\psi}_\delta\}
+\rho_\delta -2\varepsilon 
>  \mathop{\rm sup}_{\tilde{\Sigma}_\delta}\{ \tilde{U}-\varepsilon \tilde{\psi}_\delta\}
\end{eqnarray*}
for $\varepsilon$ small enough.
Since $\tilde{U}$ is 1-periodic, $\tilde{\psi}_\delta \leq 0$
on $\Omega_\delta\supset [0,1]^N$ and  $\tilde{\psi}_\delta \geq 0$ on $\Omega_\delta^C,$
it follows
\begin{eqnarray*}
&&\mathop{\rm sup}_{\R^N}\{ \tilde{U}-\varepsilon \tilde{\psi}_\delta\}
= \mathop{\rm max}_{\Omega_\delta}\{ \tilde{U}-\varepsilon \tilde{\psi}_\delta\}
=  \tilde{U}(\tilde{x}_\delta)-\varepsilon \tilde{\psi}_\delta(\tilde{x}_\delta)
\quad \text{with $\tilde{x}_\delta\in \Omega_\delta\cap \tilde{\Sigma}_\delta^C.$}
\end{eqnarray*}
Moreover
\begin{eqnarray*}
&&\tilde{U}(\tilde{x}_\delta)
\geq \tilde{U}(\tilde{x})
-\varepsilon \tilde{\psi}_\delta(\tilde{x})+
\varepsilon \tilde{\psi}_\delta(\tilde{x}_\delta)
\geq  {m}_{\eta\mu}-2\varepsilon.
\end{eqnarray*}
The claim is proved for $\varepsilon$ small enough.

Since  $\tilde{U}(\tilde{x}_\delta)>0,$ the differential
inequality holds in~\eqref{HJstatio111}
in the viscosity sense at $\tilde{x}_\delta.$
Using $\varepsilon \tilde{\psi}_\delta$ as a test-function for $\tilde{U},$ we obtain
\begin{eqnarray*}
\mathop{\rm inf}_{\theta\in\Theta}
\{-{\rm trace}(\tilde{A}_\theta(\tilde{x}_\delta)D^2 
\tilde{\psi}_\delta(\tilde{x}_\delta))\} -C|D\tilde{\psi}_\delta(\tilde{x}_\delta)|
\leq 0,
\end{eqnarray*}
which contradicts~\eqref{sur-sol-stricte}.

Therefore, there exists $\hat{x}_\delta\in \Sigma_\delta$
such that $\overline{U}(\hat{x}_\delta)={m}_{\eta\mu}.$ Letting $\delta\to 0$
and extracting subsequences if necessary, we can find 
$\hat{x}\in{\rm argmax}\,\overline{U}\cap \Sigma.$

\noindent{\it Step 2. Up to replace $\tilde{u}$ by an accumulation
point as in Lemma~\ref{Pvindept}, we may assume that $P_{\eta\mu}[\tilde{u}]$
achieves its maximum at $(\hat{x},1),$ $\hat{x}\in\Sigma.$}
From the previous step, we have $\overline{U}(\hat{x})={m}_{\eta\mu}$
for some $\hat{x}\in\Sigma.$ By definition of the half-relaxed
limit, there exists $t_n\to +\infty$ and $x_n\to \hat{x}$
such that $U(x_n,t_n)\to {m}_{\eta\mu}.$ Let $\hat{t}_n=t_n-1.$
Up to extract subsequences as in the proof of Lemma~\ref{Pvindept},
we may assume that $\tilde{u}(x,t+\hat{t}_n)$
converges uniformly in $W^{1,\infty}(\T^N\times [0,+\infty))$
to some function $\hat{u}.$ Therefore
$P_{\eta\mu}[\tilde{u}](x,t+\hat{t}_n)$ converges uniformly to $P_{\eta\mu}[\hat{u}](x,t).$ 
It follows 
\begin{eqnarray*}
P_{\eta\mu}[\tilde{u}](x_n,\hat{t}_n+1)=U(x_n,t_n)
\to P_{\eta\mu}[\hat{u}](\hat{x},1)=  {m}_{\eta\mu}.
\end{eqnarray*}
The functions $\hat{u},$ $ P_{\eta\mu}[\hat{u}]$ inherit the
properties of $\tilde{u},$ $P_{\eta\mu}[\tilde{u}]$ respectively
and it is sufficient
to prove the convergence of $\hat{u}$ to obtain the convergence
of $\tilde{u}$ and $u.$
\end{proof}

\begin{proof}[Proof of Proposition~\ref{ell-sur-sol}]
Since $\Sigma\not=\emptyset,$ by translation, we can assume without
loss of generality that $0\in \tilde{\Sigma},$ where $\tilde{\Sigma}\subset\R^N$
is a coset representative of $\Sigma\in\T^N.$
Let $\delta >0$ and $\Sigma_\delta=\{{\rm dist}(\cdot, \Sigma)\leq \delta\}.$
From~\eqref{inver_dege}, we have
\begin{eqnarray*}
&& \mathop{\rm inf}_{\tilde{\Sigma}_\delta^C} |\tilde{\sigma}_\theta(x)x|^2
= \mathop{\rm inf}_{\tilde{\Sigma}_\delta^C} \nu(x)|x|^2 =:\nu_\delta >0.
\end{eqnarray*}
We then consider the classical smooth test function which is used
to prove the strong maximum principle, that is
\begin{eqnarray*}
\tilde{\psi}_\delta (x)= e^{-\gamma_\delta r_\delta^2}-e^{-\gamma_\delta |x|^2},
\end{eqnarray*}
where we fix $r_\delta> \sqrt{N},$ $\Omega_\delta:= B(0,r_\delta)$
and $\gamma_\delta >0$ will be chosen later. 
We have
$\tilde{\psi}_\delta <0$ in $B(0,r_\delta)\supset [0,1]^N,$ $\tilde{\psi}_\delta \geq 0$ 
in $B(0,r_\delta)^C$ and $-1< \tilde{\psi}_\delta\leq  e^{-\gamma r_\delta^2}.$

For $x\in \tilde{\Sigma}_\delta^C\cap B(0,r_\delta),$ using~\eqref{H1_dege}, we have
\begin{eqnarray*}
&& -{\rm trace}(\tilde{\sigma}_\theta(x)\tilde{\sigma}_\theta(x)^T D^2\tilde{\psi}_\delta(x))
-C|D\tilde{\psi}_\delta(x)|\\
&=& 
2\gamma_\delta e^{-\gamma_\delta |x|^2} \left(
2\gamma_\delta |\tilde{\sigma}_\theta(x)x|^2 
- {\rm trace}(\tilde{\sigma}_\theta(x)\tilde{\sigma}_\theta(x)^T)
- C|x|\right)\\
&\geq& 
2\gamma_\delta e^{-\gamma_\delta |x|^2} \left(2\gamma_\delta \nu_\delta -C^2-Cr_\delta\right) >0
\end{eqnarray*}
if $\gamma_\delta$ big enough.
Therefore~\eqref{sur-sol-stricte} holds.
\end{proof}

\begin{proof}[Proof of Proposition~\ref{cor-degenere}]
For $\delta >0$ and  $\Sigma_\delta=\{{\rm dist(\cdot,\Sigma)\leq \delta}\},$
we define 
\begin{eqnarray*}
&&
K_\delta:=\bigcup_{x\in \overline{\Sigma_\delta^C}, \theta\in\Theta} 
{\rm ker}(\sigma_\theta(x))\cap \mathbb{S}^{N-1}
\subset
K_0:=\bigcup_{x\in \Sigma^C, \theta\in\Theta} 
{\rm ker}(\sigma_\theta(x))\cap \mathbb{S}^{N-1}.
\end{eqnarray*}
Using~\eqref{cont-x-theta}, we check easily
that $K_\delta$ is a compact subset of $\mathbb{S}^{N-1}.$
Since  $K_0\not= \mathbb{S}^{N-1}$ by~\eqref{non-deg-sigma},
there exists $\xi_\delta\in \mathbb{S}^{N-1}$ and $\epsilon_\delta >0$
such that
\begin{eqnarray}\label{c543}
\mathcal{C}_\delta\cap K_\delta =\emptyset,
\quad \text{with }
\mathcal{C}_\delta:=\{\zeta\in \mathbb{S}^{N-1}: 
\langle\zeta, \xi_\delta \rangle\geq 1-\epsilon_\delta\}.
\end{eqnarray}
For $\lambda >0,$
let $y_\delta=\lambda \xi_\delta\in\R^N.$ We have,
for all $x\in [0,1]^N,$
\begin{eqnarray*}
\langle \frac{y_\delta -x}{|y_\delta -x|}, \xi_\delta\rangle
= \frac{\lambda}{|\lambda \xi_\delta-x|}-
\frac{\langle x,\xi_\delta\rangle}{|\lambda \xi_\delta-x|}
\geq \frac{\lambda}{\lambda +\sqrt{N}}- \frac{\sqrt{N}}{\lambda -\sqrt{N}}
\geq 1-\epsilon_\delta
\end{eqnarray*}
for $\lambda=\lambda_{\delta}$ big enough. Therefore
$\{\frac{y_\delta -x}{|y_\delta -x|} : x\in [0,1]^N\}\subset \mathcal{C}_\delta.$
Using~\eqref{c543},~\eqref{cont-x-theta} and the periodicity
of the coset representatives $\tilde{\sigma}_\theta,$ $\tilde{\Sigma}$
of $\sigma_\theta,$ $\Sigma,$  it follows that
\begin{eqnarray*}
\nu_\delta:=\mathop{\rm inf}_{x\in  \overline{\tilde{\Sigma}_\delta^C\cap [0,1]^N}, \theta\in\Theta}
|\tilde{\sigma}_\theta(x)(y_\delta-x)|>0.
\end{eqnarray*}

For $x\in \R^N,$ we define
\begin{eqnarray*}
\phi(x)=\phi_\delta(x):= -e^{\gamma |x-y_\delta|^2-\gamma R},
\quad R:= 2|y_\delta|^2+2N+1, \ \gamma >0.
\end{eqnarray*}
Notice that $\phi$ is smooth on $\R^N$ and $-1< \phi < 0$
for all $\gamma >0.$ We have, for all $x\in \tilde{\Sigma}_\delta^C\cap [0,1]^N,$
\begin{eqnarray*}
&& -{\rm trace}(\tilde{\sigma}_\theta(x)\tilde{\sigma}_\theta(x)^T D^2\phi(x))
-C|D\phi(x)|\\
&=&2\gamma |\phi(x)|\left(
{\rm trace}(\tilde{\sigma}_\theta(x)\tilde{\sigma}_\theta(x)^T)
+2\gamma \,{\rm trace}(\tilde{\sigma}_\theta(x)\tilde{\sigma}_\theta(x)^T(x-y_\delta)\otimes (x-y_\delta))
-C |x-y_\delta|\right)\\
&\geq& 2\gamma |\phi(x)|(2\gamma \nu_\delta^2-C(r+|y_\delta|))>0
\end{eqnarray*}
for $\gamma=\gamma_{\delta, r}$ big enough.
Therefore $\phi$ is a smooth supersolution of the equation
in~\eqref{sur-sol-stricte} in $\tilde{\Sigma}_\delta^C\cap (0,1)^N.$

We now define $\tilde{\psi}_\delta,$ $\Omega_\delta$ on the following
way. 
We set $\tilde{\psi}_\delta(x) =\phi (x)$ for $x\in \tilde{\Sigma}_{\delta/2}^C\cap [0,1]^N$
Now, from~\eqref{F-bord-tore}, we have
$\{{\rm dist}(\cdot,\partial [0,1]^N)\leq \delta/4\}\cap \tilde{\Sigma}_{\delta/2}^C=\emptyset$
so we can extend $\tilde{\psi}_\delta$ in a smooth way in $[0,1]^N$ such that
$\tilde{\psi}_\delta(x) =0$ for $x\in \{{\rm dist}(\cdot,\partial [0,1]^N)\leq \delta/4\}\cap [0,1]^N$
and $|\tilde{\psi}_\delta|\leq 1$ in $[0,1]^N.$ We then extend $\tilde{\psi}$ outside $[0,1]^N$
by $0.$  We set $\Omega_\delta:= \{{\rm dist}(\cdot,\partial [0,1]^N)< \delta/4\}.$
It is straightforward that the function $\tilde{\psi}_\delta$ satisfies~\eqref{sur-sol-stricte}.
\end{proof}





\end{document}